\documentclass{amsart} 
\usepackage{amssymb,amsmath,amsfonts,color,verbatim}
\usepackage{kbordermatrix}

\begin{document}  
\bibliographystyle{plain}

\newtheorem{thm}{Theorem}[section]
\newtheorem{lem}[thm]{Lemma}
\newtheorem{prop}[thm]{Proposition} 
\newtheorem{conj}[thm]{Conjecture}
\newtheorem{quest}[thm]{Question}

\theoremstyle{definition}
\newtheorem{defn}[thm]{Definition}

\newtheorem*{namedtheorem}{Theorem \ref{thm:main}}

\newcommand{\h}{\ensuremath{\mathbb{H} } }

\newcommand{\Hom}{\ensuremath{\operatorname{Hom} } }
\newcommand{\Isom}{\ensuremath{\operatorname{Isom} } }
\newcommand\E {\ensuremath{\mathcal{E} } }
\newcommand{\ad}{\ensuremath{\operatorname{ad} } }
\newcommand{\Ad}{\ensuremath{\operatorname{Ad} } }
\newcommand{\so}{\ensuremath{\operatorname{\mathfrak{so} } } }
\newcommand{\sech}{\ensuremath{\operatorname{sech} } }

\def\urltilda{\kern -.15em\lower .7ex\hbox{\~{}}\kern .04em}

\numberwithin{equation}{section}
\numberwithin{figure}{section}
\numberwithin{table}{section}

\newcommand{\R}{\ensuremath{\mathbb{R} } }
\newcommand{\skewm}{\ensuremath{\operatorname{skew} } }
\newcommand{\rotm}{\ensuremath{\operatorname{rot} } }
\newcommand{\tr}{\ensuremath{\operatorname{tr} } }
\newcommand{\curl}{\ensuremath{\operatorname{curl} } }
\newcommand{\del}{\nabla }

\renewcommand{\kbldelim}{(}
\renewcommand{\kbrdelim}{)}

\def\lap{\triangle}
\def\grad{{\nabla}}
\def\homeo{{\approx}}
\def\W{{Weitzenb\"ock  }}


\def\bA{\textbf{A}}
\def\bB{\textbf{B}}
\def\bC{\textbf{C}}
\def\bD{\textbf{D}}
\def\bE{\textbf{E}}
\def\bF{\textbf{F}}
\def\bG{\textbf{G}}
\def\bH{\textbf{H}}


\title{Local Rigidity of Hyperbolic Manifolds with Geodesic Boundary}
\author{Steven P. Kerckhoff and Peter A. Storm} \date{\today}

\begin{abstract}
Let $W$ be a compact hyperbolic $n$-manifold with totally geodesic boundary.  We prove that if $n>3$ then the holonomy representation of $\pi_1 (W)$ into the isometry group of hyperbolic $n$-space is infinitesimally rigid.  
\end{abstract}

\thanks{Kerckhoff and Storm were partially supported by NSF grants DMS-0605151 and DMS-0741604, respectively.  Storm also received support from the Roberta and Stanley Bogen Visiting Professorship at Hebrew University. } 
\maketitle

\section*{}\label{sec:introduction}
\setcounter{section}{1} 

Starting with a complete hyperbolic manifold $X$, when is it possible to deform its hyperbolic
structure?  If $X$ is a finite area 
surface, its metric always has nontrivial deformations 
through complete hyperbolic metrics. 
However, in all higher dimensions, the work of Calabi, Weil, and Garland combine to prove that there are no nontrivial 
deformations through complete metrics when $X$ has finite volume \cite{Cal,We,Garland}.  
In search of flexibility, one is led to study infinite volume hyperbolic manifolds.

In dimension $3$ infinite volume hyperbolic manifolds have been studied for many years.
A particularly well-understood class of such manifolds are the convex cocompact ones. 
They are topologically the interior of a compact $3$-manifold with boundary consisting of surfaces of genus at least $2$.  
Convex cocompact hyperbolic $3$-manifolds have large deformation spaces parametrized by conformal structures on their boundary surfaces inherited from the sphere at infinity.
In higher dimensions the situation becomes more mysterious, but one expects to find less flexibility than in dimension $3$.

This article's research began by looking for a large class of infinite volume hyperbolic manifolds, 
present in all dimensions, and rigid in dimensions greater than $3$.  In high dimensions it is difficult 
to construct interesting infinite volume hyperbolic manifolds.  One method is to begin with a closed 
hyperbolic manifold (often constructed by arithmetic tools) that contains an 
embedded totally geodesic hypersurface and 
cut along the hypersurface to obtain a compact hyperbolic manifold $W^n$ with totally geodesic boundary.  
There is a canonical extension of such a structure to a complete, infinite volume hyperbolic 
metric on an open manifold $X^n$ without boundary which is diffeomorphic to the interior of $W^n$. 
This is the class of manifolds we will study.

Associated to any hyperbolic manifold is a representation, called the \emph{holonomy representation}, of its fundamental group
in the group $\Isom (\h^n)$ of isometries of $n$-dimensional hyperbolic space.
When it is a compact hyperbolic manifold $W$ with totally geodesic boundary,
the representation is discrete and faithful.  If we denote the image of 
the representation by
$\Gamma$, then $W = K/\Gamma$ where $K$ is a closed subset of $\h^n$ bounded by a
collection of geodesic hyperplanes that map onto the boundary of $W$.
The quotient $\h^n/\Gamma$ of all of hyperbolic space by $\Gamma$ is a complete, infinite volume hyperbolic $n$-manifold $X$ that we call the extension of $W$.  It is diffeomorphic to the interior of $W$ and contains $W$ as a compact, convex subset. In particular,
$\pi_1 X = \pi_1 W$ and the holonomy representations of $X$ and $W$ are the same.

The holonomy representation $\rho$ of $\pi_1 W$ is uniquely determined, up to conjugation in $\Isom (\h^n)$, by the hyperbolic structure on $W$.  The space $\Hom(\pi_1 W, \Isom (\h^n))$ is topologized by the 
compact-open topology.  The holonomy representation $\rho$ is said to be \emph{locally rigid}  
if a neighborhood of it in $\Hom(\pi_1 W, \Isom (\h^n))$ consists entirely of conjugate 
representations. In this case, we will refer to the hyperbolic manifold $W$ (or $X$) as locally rigid.  Note that nearby representations of $\pi_1 W$ need not correspond to
hyperbolic structures with geodesic boundary.

The main result of this paper is:

\begin{thm} \label{thm:intromain}
Let $W$ be  a compact hyperbolic $n$-manifold with totally geodesic boundary.  If $n>3$ then the holonomy representation of $W$ is locally rigid.
\end{thm}

We will actually prove (Theorem \ref{thm:main}) the slightly stronger result that such 
hyperbolic manifolds $W$ are \emph {infinitesimally rigid}, a concept that will be discussed in Section \ref{sec:algebraic preliminaries}.

Theorem \ref{thm:intromain} was known only in a couple special cases prior to this paper.  Initially, 
M. Kapovich observed that Mostow rigidity could be applied to prove an orbifold version of Theorem \ref{thm:intromain} 
for the group generated by reflections in $119$ of the $120$ walls ($3$-cells) of the hyperbolic right-angled 
$120$-cell in $\mathbb{H}^4$.  This group produces an orbifold quotient of $\mathbb{H}^4$ with totally geodesic boundary
in a suitable orbifold sense.  Again starting with the hyperbolic $120$-cell, Aougab and the second author 
verified Theorem \ref{thm:intromain} for a specific group generated by reflections in $96$ walls of the hyperbolic 
right-angled $120$-cell \cite{St13}.  These $96$ walls were chosen so the complementary $24$ walls form a 
maximal collection of pairwise disjoint walls. 

We remark that, at least when $n = 4$, it is essential that the manifolds $W$ 
with totally geodesic boundary be compact.  
Specifically, the hypotheses of Theorem \ref{thm:intromain} cannot be loosened to permit 
$W$ to have finite volume.  With such a hypothesis Theorem \ref{thm:intromain} is false.  This
is due to the $4$-dimensional example constructed in \cite{KS}.  There, the authors
study in detail deformations of reflection groups derived from the $4$-dimensional hyperbolic $24$-cell.
These deformations involved a $4$-dimensional and infinite volume analog of $3$-dimensional hyperbolic Dehn filling. 
The proof of Theorem \ref{thm:intromain} shows that, in a suitable sense, the deformations of \cite{KS} are the only type possible
for $4$-manifolds with totally geodesic boundary.  In particular, we conjecture that this is a purely $4$-dimensional phenomenon and that Theorem \ref{thm:intromain} remains true
in dimensions $\geq 5$ even if $W$ is allowed to be non-compact with finite volume. 
We will comment on this further in Section \ref{sec:conjectures}.

\medskip

One of the central ideas in the proof of Theorem \ref{thm:intromain} is the interaction
between the deformations of the holonomy representation of the $n$-dimensional manifold
$W$ and those of its $(n-1)$-dimensional geodesic boundary.
Starting with a hyperbolic $n$-manifold $W^n$ with totally geodesic boundary $M^{n-1}$,  we
can restrict the holonomy representation $\rho$ of $\pi_1 W$ to the fundamental group
of any component of $M$.  For simplicity of discussion, we 
will assume that $M$ is connected.  We denote the restricted holonomy representation by 
$\rho_M$.  Since $M$ is an $(n-1)$-dimensional hyperbolic manifold it is the quotient
of a geodesic hyperplane which is preserved by the isomorphic image of 
$\pi_1 M$ under $\rho_M$.   Now glue two copies
of $W$ along the boundary to 
obtain a hyperbolic structure on the double of
$W$.  The double is a closed manifold which can't have any nontrivial deformations when $n\geq 3$.  
This implies, in these dimensions,
that there are no nontrivial deformations of the holonomy representation of $W$ 
through hyperbolic structures with totally geodesic boundary, ie. deformations 
where the restriction $\rho_M$ continues to preserve a geodesic hyperplane.

Let $\Gamma$ denote the
image $\rho_M(\pi_1 M)$ in $\Isom(\h^n)$ and consider the $n$-dimensional hyperbolic
manifold $X = \h^n/\Gamma$.  Topologically, $X$ is a product $M \times \R$.  Geometrically, it has infinite volume and
contains a totally geodesic submanifold isometric to $M$ that completely determines its 
geometry.  We will call such an $n$-dimensional hyperbolic manifold \emph{Fuchsian}.  

In all dimensions, Johnson and Millson \cite{JMbending} found Fuchsian hyperbolic
manifolds $X = M \times \R$ with
nontrivial deformations.  These deformations depend on the existence of 
a codimension $1$
totally geodesic submanifold in $M.$  This geodesic submanifold allows one to deform
the holonomy representation of $X$ in $\Isom(\h^n)$ using the process of bending.  (See Example 8.7.3 of \cite{Th} or \cite{JMbending} for an explanation of bending.) On the other hand, specific 
examples of $4$-dimensional Fuchsian manifolds without deformations were discovered by Kapovich \cite{Kap2} and 
Scannell \cite{Scannell}.   

The flexibility or rigidity of an $n$-dimensional Fuchsian manifold $X = M \times \R$ is really a question
about the $(n-1)$-dimensional hyperbolic manifold $M$.
We will say that $M^{n-1}$ is either 
\emph{deformable} or \emph{locally rigid} in $\Isom(\h^n)$ depending on whether or not
its associated Fuchsian manifold $X^n$ has nontrivial deformations.  
Although both Kapovich and Scannell
construct infinite families of rigid examples, it remains unclear
what geometric or topological condition on $M^{n-1}$ would imply local rigidity in $\Isom(\h^n)$,
even when $n=4$.  
This is a very interesting question about which little is known.  We do not address 
it here (though see Theorem \ref{harmonicbends}).

The discussion above implies that, for the holonomy representation $\rho$ of 
a hyperbolic $n$-manifold $W$ with geodesic boundary to have nontrivial deformations, its
$(n-1)$-dimensional boundary $M$ must be deformable in $\Isom(\h^n)$.  
This suggests
the possibility that Theorem \ref{thm:intromain} is true simply because any closed 
$(n-1)$-dimensional hyperbolic manifold $M$ which is the totally geodesic boundary
of an $n$-dimensional manifold $W$ is, in fact, locally rigid in $\Isom(\h^n)$.
However, examples of Gromov and Thurston show this is definitely not the case \cite{GromThur}. 

For each dimension $n>2$ they construct 
an infinite number of closed hyperbolic $n$-manifolds $V$ with the following properties:  
$V$ has a codimension $1$ embedded, totally geodesic submanifold $M$ which itself has a 
codimension $1$ embedded, totally geodesic submanifold $P$ (so $P$ is codimension $2$ in $V$).  
For each $V$, cutting along $M$ provides an $n$-dimensional hyperbolic manifold $W$ with
geodesic boundary $M$.  The existence of the geodesic submanifold $P$ in $M$ implies
that $\rho_M$ is deformable in $\Isom(\h^n)$ by bending.  

This point of view provides another interpretation of Theorem \ref{thm:intromain}.
It says that, for $n\ge 4$, when a closed $(n-1)$-dimensional hyperbolic manifold $M$
is the boundary of a compact $n$-dimensional manifold $W$, even when $M$ is deformable
in $\Isom(\h^n)$, this deformation does not extend to $W$.

\medskip

In dimension $2$ it is easy to construct compact hyperbolic surfaces of any genus with totally geodesic boundary.   Thurston's  hyperbolization theorem for Haken manifolds implies that 
any compact, irreducible $3$-manifold with genus $g \geq 2$ boundary has a hyperbolic 
structure with totally geodesic boundary as long as it contains no nontrivial annuli or tori. 
In higher dimensions it is much more difficult to find explicit examples.  It follows from a 
theorem of Wang \cite{Wa} that in any fixed dimension $n > 3$ the number $N(V)$ of 
(isometry classes of) compact hyperbolic $n$-manifolds with nonempty totally geodesic boundary 
and volume less than $V$ is finite.
(It becomes infinite in dimension $3$.) Nonetheless, $N(V)$ grows faster than $V^{aV}$ for some $a>0$.  
In particular, in \cite{Rat2} Ratcliffe and Tschantz construct an explicit infinite family of compact hyperbolic $4$-manifolds with nonempty totally geodesic boundary that provides such a lower bound.  
For general $n\geq 4$ it is shown in \cite{BGLM} that the number of \emph {closed} hyperbolic $n$-manifolds 
of volume at most $V$ grows at least this fast.  But their examples (which are covering spaces 
of a single example) all contain embedded totally geodesic hypersurfaces along which one 
can cut to obtain a sufficient number of manifolds with totally geodesic boundary.  
(One needs also to bound the number of pieces that become isometric after cutting. 
We will not provide that argument here.)   Thus, while the structures covered by
Theorem \ref{thm:intromain} are not as prevelant as those in lower dimensions, they are,
nonetheless, numerous.

\medskip

A compact hyperbolic $n$-manifold $W$ with totally geodesic boundary has an extension $X$ which is an infinite volume, complete hyperbolic manifold without boundary.  Since 
$\pi_1 X = \pi_1 W$ and the images of the holonomy representations of $X$ and $W$
are the same, Theorem \ref{thm:intromain} can be rephrased in terms of the local
rigidity of the holonomy representation of $X$.  This, in turn, can be described
directly in Riemannian geometric terms.  

The hyperbolic structure on $X$ is convex cocompact, which means that it contains an $n$-dimensional convex, compact submanifold $C$ that is homotopy equivalent to $X$.
(This concept is discussed in more detail in Section \ref{sec:hyperbolic preliminaries}.)
When $X$ is the extension of a hyperbolic manifold $W$ with geodesic boundary, $C$ can 
be taken to equal $W$.  Conversely, if $X$ has such an $n$-dimensional convex, compact submanifold $C$ with totally geodesic boundary, it is the extension of $C$.   We say that
$X$ has \emph {Fuchsian ends} in this case.  Each component of the complement $X-C$ is diffeomorphic 
to the product $\partial C \times (0,\infty)$. The term 
``Fuchsian end'' reflects the fact that each component of $\partial C$ is an $(n-1)$-dimensional 
hyperbolic manifold and hence is determined by a discrete group of isometries of 
$\h^{n-1}.$   (This term also applies to the case when $X$ is actually Fuchsian, but
we are ruling out this case by the assumption that $C$ is $n$-dimensional.)
Thus we have a $1{-}1$ correspondence between 
(non-Fuchsian) hyperbolic manifolds with Fuchsian ends and compact hyperbolic manifolds 
with totally geodesic boundary.  

If $X$ is deformable, any nearby structure will still be convex cocompact.  When 
$n \geq 3$ any nontrival deformation cannot still have Fuchsian ends 
(by local rigidity of the double of $C$).
If $X$ is $3$-dimensional, it has a large deformation space, parametrized by conformal structures on surfaces at infinity (which are diffeomorphic to 
$\partial C$).  However, for $n >3$, quite the opposite is true and there are no nearby non-isometric hyperbolic structures on $X$.  This is the Riemannian version of 
Theorem \ref{thm:intromain} which we state as a separate theorem below.
In Section \ref{sec:hyperbolic preliminaries} the relationship between the two theorems will be discussed in more detail. 

\begin{thm}\label{thm:intro}
Let $X$ be an infinite volume convex cocompact complete hyperbolic $n$-manifold without boundary. 
Assume that $X$ has Fuchsian ends but is not Fuchsian.  If $n>3$ then $X$ has no nontrivial deformations.  Specifically, if $g_n$ is a sequence of complete hyperbolic metrics on $X$ converging to the given metric in the compact-$\mathcal{C}^\infty$ topology, then for sufficiently large $n$ the Riemannian manifolds $(X,g_n)$ are all isometric.
\end{thm}

Theorem \ref{thm:intro} can be viewed as a generalization of the local 
rigidity theorems of Calabi and Weil to the infinite volume setting.  It is natural to ask whether 
a generalization of the global rigidity theorem of Mostow holds in this case; in other words, does 
the smooth manifold underlying $X$ support a unique complete hyperbolic structure?  
(See Question \ref{quest:Mostow}.) 
Our methods cannot answer this global question.


\section{Preliminaries}\label{sec:preliminaries}

\subsection{Hyperbolic preliminaries}\label{sec:hyperbolic preliminaries}

In this subsection we will define some useful terms from hyperbolic geometry, including some
already mentioned in the introduction. 
We will use the standard convention that
a \emph{manifold} has empty 
boundary, unless explicitly
stated otherwise.  However, we will occasionally refer to a \emph{manifold without boundary} for
the sake of clarity.

A hyperbolic structure on an $n$-dimensional manifold $X$ determines a developing map
from the universal cover $\widetilde X$ of $X$ to $n$-dimensional hyperbolic space $\h^n$.
It is equivariant with respect to the action of $\pi_1 X$ on $\widetilde X$ and $\Isom(\h^n)$
on $\h^n$, inducing a representation $\rho: \pi_1 X \to \Isom(\h^n)$, called the 
\emph{holonomy representation} of $X$.  The developing map is well-defined up to 
post-multiplication by an element in $\Isom(\h^n)$, and the holonomy representation
is uniquely determined up to conjugation in $\Isom(\h^n)$ by the hyperbolic structure.  
If 
the hyperbolic structure on $X$ is complete then 
the developing map is a diffeomorphism and the holonomy representation is faithful.
In this case we can identify $\widetilde X$ with $\h^n$ by choosing a particular developing
map.  This choice 
determines $\rho$, identifies $\pi_1 X$ with $\Gamma = \rho(\pi_1 X)$, and induces
an isometry between $X$ and $\h^n/\Gamma$.  

Associated to any complete hyperbolic structure on $X$ is its convex core, which we now define.
For more information see \cite{EM}.

\begin{defn}\label{defn:convex core}
Let $X$ be a complete hyperbolic $n$-manifold, inducing an action of $\pi_1 (X)$ on $\h^n$ by isometries.  
Let $\mathcal{S}$ be the set of subsets $Y$ of $\h^n$ satisfying:
\begin{itemize}
\item $Y$ is closed and nonempty.
\item $Y$ is convex.
\item $\gamma \cdot Y = Y$ for all $\gamma \in \pi_1(X)$.
\end{itemize}
The \emph{convex hull} $\operatorname{Hull}(X) \subseteq \h^n$ is defined as the 
intersection $\cap_{Y \in \mathcal{S}} Y$.  The quotient 
\[C(X) := \operatorname{Hull}(X) / \pi_1(X) \subseteq X\]
is the convex core of $X$.
\end{defn}

The convex hull of $X$ is convex, and thus contractible.  Therefore $C(X)$ is homotopy equivalent to $X$.  
If $X$ has finite volume then $\operatorname{Hull}(X) = \h^n$ and $C(X)=X$.  When $X$ has infinite volume, the convex core is typically an $n$-dimensional submanifold with boundary, but, at one extreme, it can be all of $X$, or, 
at the other extreme, a lower dimensional submanifold.  In all cases $X$ is homeomorphic to the gluing
\[ C_{\varepsilon}(X) \cup_{\partial C_{\varepsilon}(X)} \left( \partial C_{\varepsilon} (X) \times [0,\infty) \right), \]
where $C_{\varepsilon}(X)$ is a closed 
$\varepsilon$-neighborhood of $C(X)$ in $X$.

\begin{defn}\label{defn:convex cocompact}
If the convex core of $X$ is compact then $X$ is \emph{convex cocompact}.
\end{defn}

Although we will not need it in this paper, for cultural context we mention that 
$X$ is called geometrically finite if $C_{\varepsilon}(X)$ has finite volume for some $\varepsilon>0$.  This allows 
the image of the holonomy representation to contain parabolic isometries, 
which is not possible when $X$ is convex cocompact.

\begin{defn}\label{defn:Fuchsian}
A complete hyperbolic $n$-manifold $X$ is \emph{Fuchsian} if it has a holonomy representation $\pi_1(X) \to \Isom(\h^n)$ whose image $\Gamma$ lies in a subgroup $\Isom(\mathbb{H}^m)$ for $m<n$.   We assume that the quotient $\h^m/\Gamma$ has finite volume.
\end{defn}

A Fuchsian manifold $X$ 
comes from an $m$-dimensional hyperbolic manifold $M = \h^m/\Gamma_M$ 
whose holonomy image $\rho_M(\pi_1(M)) = \Gamma_M$ is equal to $\Gamma$.  The convex core $C(X) \subset X$ is a totally geodesic submanifold of dimension $m$, isometric to $M$, which uniquely determines
the geometry of $X$.   We will only be considering the case when $m = n-1$.
Then $C(X)$ has codimension $1$ and $X$ is isometric to the warped product metric on $C(X) \times \mathbb{R}$ given by
\begin{equation}\label{eqn:warped}
(\cosh^2 t) g \oplus dt^2, 
\end{equation}
where $g$ is the metric on $C(X)$.  

The next definition is used to describe hyperbolic structures whose ends are isometric to 
a Fuchsian structure with codimension $1$ convex core. 

\begin{defn}\label{defn:Fuchsian ends}
Let $X$ be a complete hyperbolic $n$-manifold without boundary.  Then $X$ \emph{has Fuchsian ends} 
and is not Fuchsian 
if its convex core $C(X)$ is an $n$-manifold with totally geodesic boundary, and some $\varepsilon$-tubular neighborhood of $C(X)$ has finite volume. 
\end{defn}

In general, this terminology allows $C(X)$ to have finite volume ends. 
We will avoid this possibility by assuming that $X$ is convex compact.  Each 
end of $X$ is geometrically determined by a component $S$ of $\partial C(X)$. The end beginning at $S$ is
isometric to $S \times [0,\infty)$ equipped with the Fuchsian warped product metric (\ref{eqn:warped}). 

Given a hyperbolic manifold $W$ with totally geodesic boundary, we can form a manifold $X$ with Fuchsian ends and empty boundary 
by simply isometrically gluing ends $\partial W \times [0,\infty)$ onto the boundary of $W$ equipped with the Fuchsian warped product metric (\ref{eqn:warped}).  The convex core of $X$ will be $W$.  This relationship is canonical: a hyperbolic manifold with nonempty totally geodesic boundary determines a unique manifold with Fuchsian ends and vice versa.

Given this tight relationship, it is natural to ask whether it's possible to phrase Theorem \ref{thm:intro} in terms of deformations of compact hyperbolic manifolds with totally geodesic boundary.  This is possible, but requires some care.  Let $X$ be a convex cocompact hyperbolic manifold with Fuchsian ends, dimension $n \ge 4$, and $n$-dimensional convex core $C$. 
Then simply replacing $X$ with $C$ everywhere in Theorem \ref{thm:intro} produces a false statement.  There are many hyperbolic metrics on the underlying manifold with boundary 
which is
diffeomorphic to $C$.  One could, for example, take a closed $\varepsilon$-neighborhood $C_\varepsilon$ of $C$ inside $X$.  The manifold $C_\varepsilon$ would be diffeomorphic but not isometric to $C$ because the boundary of $C_\varepsilon$ would be strictly convex rather than totally geodesic.  However, it is clear that such a deformation is, in some sense, trivial.  One can rule out such trivial deformations, yielding the following restatement of Theorem \ref{thm:intro}.

\begin{thm}\label{thm:intro with boundary}
Let $C$ be a compact hyperbolic $n$-manifold with nonempty totally geodesic boundary.  Let $X$ be the 
manifold with Fuchsian ends determined by $C$.  If $n>3$ then $C$ has no nontrivial deformations in the following sense:
if $g_n$ is a sequence of hyperbolic metrics on $C$ converging to the given metric in the $\mathcal{C}^\infty$ topology,
then for sufficiently large $n$ the Riemannian manifolds $(C,g_n)$ admit isometric embeddings into $X$, and 
the boundaries of the images converge in the $\mathcal{C}^{\infty}$ topology to $\partial C$.
\end{thm}

The proof of Theorem \ref{thm:intro with boundary} uses Theorem \ref{thm:intro}, the fact that the manifolds 
$(C,g_n)$ admit analytic thickenings to open manifolds \cite{CEG}, and that $\mathcal{C}^\infty$ convergence 
guarantees the boundaries are nearly totally geodesic.  As it is not central to this paper, the details are omitted.

\subsection{Algebraic preliminaries}\label{sec:algebraic preliminaries}

The holonomy representation leads to an algebraic description of the
deformation theory of hyperbolic structures.  In this subsection we briefly outline 
this algebraic structure.  For more detailed presentations of this material see \cite{We,We2,Rag,Gold1,HK1,HK3}.

Let $X$ be a complete hyperbolic $n$-manifold with universal cover $\widetilde X$. Let $G$ denote 
the group of isometries of $\h^n$. Choose an isometric identification of $\widetilde X$ with $\h^n$ and 
let $\rho: \pi_1(X) \to G$ be the associated
holonomy representation.  Denote by $\Gamma$ the image group $\rho(\pi_1(X))$ so that we can
identify $X$ with the quotient $\h^n/\Gamma$.  The representation $\rho$ is a point in the
space of representations 
$\Hom(\pi_1(X),G)$, which is equipped with the compact-open topology.  A sequence of hyperbolic metrics 
$g_n$ on $X$ determines a sequence of holonomy representations, each of which is only determined 
up to conjugacy.  However, it is a standard fact that if the metrics converge smoothly then it is possible to choose a convergent sequence of holonomy representations.

\begin{lem}\label{lem:Riem2holonomy}
Suppose $g_n$ is a sequence of hyperbolic metrics on $X$ converging in the compact-$\mathcal{C}^\infty$ 
topology to the initial metric $g$.  Then it is possible to choose holonomy representations $\rho_n$ of $(X,g_n)$ 
such that $\rho_n \to \rho$ in $\Hom(\pi_1(X),G)$.
\end{lem}

\begin{defn}
A hyperbolic manifold $X$ is \emph{locally rigid} if a holonomy representation $\rho$ for $X$ has a 
neighborhood $U \subseteq \Hom(\pi_1(X),G)$ such that any representation in $U$ is conjugate to $\rho$.
\end{defn}

By Lemma \ref{lem:Riem2holonomy}, if $X$ is locally rigid then $X$ is locally rigid in the sense 
of Theorem \ref{thm:intro}.
The converse is not generally true.  However,
it is true when $X$ is compact or, more generally, when it is convex cocompact. 
This converse direction will not be used here, and is not pursued further. 

Suppose $\rho_t$ is a smooth path in $\Hom(\pi_1(X),G)$ such that $\rho_0 = \rho$.  Then for each $\gamma \in \pi_1(X)$,  
$\rho_t(\gamma)$ is a smooth path in $G$ and 
the derivative 
\[ \frac{d}{dt} \left( \rho_t (\gamma) \cdot \rho(\gamma)^{-1}\right)\] 
at $t=0$ determines
an element in the Lie algebra $\mathfrak{g}$, where the latter is identified with the
tangent space to $G$ at the identity.  This determines a set 
map $z: \pi_1(X) \to  \mathfrak{g}$.  The fact that $\rho_t$ is a homomorphism implies that
$z(\gamma_1 \gamma_2) = z(\gamma_1) + \rho(\gamma_1) z(\gamma_2) \rho(\gamma_1)^{-1}.$   
Define a $\pi_1(X)$-action on $\mathfrak{g}$ by $\gamma \cdot V := \Ad (\rho(\gamma))\cdot V$, 
where $\Ad$ denotes the adjoint action.  The 
previous equation becomes
$z(\gamma_1 \gamma_2) = z(\gamma_1) + \gamma_1 \cdot z(\gamma_2).$
A map satisfying this equation is defined to be a \emph{$1$-cocycle} for the group cohomology of 
$\pi_1(X)$ with coefficients in $\mathfrak{g}$ twisted by the above $\pi_1(X)$-action.  
The group of such cocycles is denoted by $  Z^1(\pi_1(X); \mathfrak{g}_{\Ad  \rho} ).$

Given a smooth path $g_t \in G$ so that $g_0 = e$ one can define a path of representations 
$\rho_t =  g_t \rho g_t^{-1}$.  The associated cocycle $z$ in this case equals 
$z(\gamma) = V - \gamma \cdot V$ where $V \in \mathfrak{g}$ is the vector tangent to $g_t$ at 
$t=0$.  The right hand side of this equation defines a cocycle for any $V \in \mathfrak{g}$.  
The set of such cocycles, called coboundaries, is denoted by 
$  B^1(\pi_1(X); \mathfrak{g}_{\Ad \rho} ).$

The \emph{cohomology group} $H^1 (\pi_1(X); \mathfrak{g}_{\Ad \rho})$ is defined to be the quotient 
\[ \frac{Z^1 (\pi_1(X); \mathfrak{g}_{\Ad \rho})}{B^1(\pi_1(X); \mathfrak{g}_{\Ad \rho} )}.\]
This quotient group can be interpreted as an algebraic description of the tangent 
space at $\rho$ to the representation space $\Hom(\pi_1(X),G)$ modulo conjugation.  Under 
certain smoothness hypotheses it is isomorphic to the actual tangent space of this quotient space.  
In any case, it is referred to as the space of \emph{infinitesimal deformations} of the
hyperbolic structure on $X$.
This leads us naturally to the next definition.

\begin{defn}\label{defn:infrig}
A hyperbolic manifold $X$ is \emph{infinitesimally rigid} if 
\[H^1 (\pi_1(X); \mathfrak{g}_{\Ad \rho}) = 0.\]
\end{defn}

Infinitesimal rigidity states that any path $\rho_t$ through the holonomy representation $\rho$ agrees to first order at $\rho$ with a geometrically trivial path consisting entirely of conjugate representations.  An application of the implicit function theorem proves the following lemma, known as Weil's lemma.

\begin{lem}\label{lem:Weil}\cite{We2}
If $X$ is infinitesimally rigid then it is locally rigid.  
\end{lem}

With this we are ready to state the main theorem of this paper.

\begin{thm}\label{thm:mainFuchs}
Let $X$ be  a convex cocompact hyperbolic $n$-manifold without boundary which has Fuchsian ends 
but is not Fuchsian.  If $n>3$ then $X$ is infinitesimally rigid.
\end{thm}

Theorem \ref{thm:mainFuchs} implies Theorem \ref{thm:intro} of the introduction, but the converse implication need not hold.  There do exist locally rigid manifolds that are not infinitesimally rigid \cite{GoldMill}.  

Despite its simple description  the cohomology group $H^1 (\pi_1(X); \mathfrak{g}_{\Ad h})$
is extremely difficult to compute.  Given an explicit presentation of $\pi_1(X)$ and a formula
for $\rho$, computations are occasionally 
possible, at least with a computer.   However, 
a direct algebraic computation for any large class of examples is generally infeasible.
We will need to use analytic and geometric tools to describe this group.

Let $X$ be a complete hyperbolic manifold which has been identified with a quotient
$\h^n/\Gamma$ where $\Gamma = \rho(\pi_1(X))$ is the image in $G$ of a particular holonomy
representation $\rho$.
Consider the product $\widetilde \E  := \mathbb{H}^n \times \mathfrak{g}$ equipped with the diagonal $\Gamma$-action.
The quotient is a $\mathfrak{g}$-bundle $\E$ over $X$.  Note that the $\Gamma$-action on $\widetilde \E$ takes a constant section to a constant section.  Thus constant sections define a flat bundle structure on $\widetilde \E$ that descends to $\E$.

The group cohomology $H^1 (\pi_1(X); \mathfrak{g}_{\Ad \rho})$ is isomorphic to the singular 
cohomology group $H^1(X;\E)$.  (This is true for any connected CW complex with fundamental group equal to 
$\pi_1(X)$.)  In turn, an extension of the usual deRham theorem shows that singular
cohomology equals deRham cohomology with values in $\E$.  The latter can be 
described as follows.  Consider the space $\Omega^{*}(\h^n; \widetilde \E)$ of $\widetilde \E$-valued differential forms on $\h^n$.  
Given a basis $\{B_i \}$ for $\mathfrak{g}$, such a $k$-form is given concretely as a finite linear 
combination $\sum B_i \otimes \theta_i$, where each $\theta_i$ is a real-valued differential $k$-form.  
Then $\Gamma$ acts on $\Omega^{*}(\h^n; \widetilde \E)$ via the adjoint action on $\widetilde \E$ and 
the usual action on differential forms. The $\widetilde \E$-valued 
forms fixed by this action form 
the space $\Omega^{*}(X;\E)$ of 
$\E$-valued differential forms on $X$.  Using the constant basis, the differential $d_\E$ can be defined on $\Omega^{*}(\h^n;\widetilde \E)$ as simply
\[ d_\E  \left( \sum B_i \otimes \theta_i \right) := \sum B_i \otimes d\theta_i .\]
On $\Omega^{*}(X;\E)$ the differential $d_\E$ is defined by applying the above definition on small open 
neighborhoods of $X$.  With this differential on 
the space of forms $\Omega^{*}(X;\E)$, we can define the deRham cohomology groups $H_{dR}^{*} (X;\E)$ as the closed $\E$-valued forms modulo the exact ones.

Of course, deRham cohomology is also difficult to compute directly.  In order to understand these groups, 
it will be necessary to use Hodge theory.  
We need a metric on the fibers of $\E$ to define the notion of a harmonic form in $\Omega^{*}(X;\E)$. 
For this a little terminology is required.  Represent $G= \Isom(\h^n)$ explicitly as the index $2$ subgroup of $\operatorname{O}(1,n)$ preserving the upper hyperboloid and let $K$ denote the maximal compact 
subgroup $\operatorname{O}(n)$ of $G$ fixing $(1,0,\ldots,0)$, 
so that $\h^n = G/K$.  Then consider the $1$-parameter 
family $g_t$ in $G$ given by the matrices of the form
\[ \begin{pmatrix}
A & 0 \\
0 & I
\end{pmatrix}, \quad \text{where} \quad A =
\begin{pmatrix}
\cosh t & \sinh t \\
\sinh t & \cosh t
\end{pmatrix}  \]
and $I$ is the $(n-1)$-dimensional identity matrix.  Taking the derivative at $t = 0$ determines a tangent vector at the identity element $e \in G$, hence an element of the Lie algebra $V \in \mathfrak{g}.$  
The family $g_t$ acts on $\h^n$ by isometries, translating along a geodesic through the basepoint 
$x_0 \in \h^n$ which is the image of $e \in G$ under the projection to $\h^n/K$.  (Viewing 
$\h^n$ as the upper hyperboloid in $\R^{1,n}$, $x_0$ is the point $(1,0,\ldots,0)$.)  Taking the derivative 
at $t=0$ determines a tangent vector based at $x_0.$  Conjugation by an element $\gamma \in G$ determines 
a family $\gamma g_t \gamma^{-1}$ that acts
on $\h^n$ by translation along a geodesic through $\gamma x_0$. 
The Lie algebra element obtained by taking the derivative of the 
conjugated family at $t=0$ equals $\Ad(\gamma)\cdot V$.  Taking the derivative of the conjugated action 
on $\h^n$ at $t=0$ determines a unit tangent vector in $T_{\gamma x_0} \h^n$. 
In this way a unit vector at a point $p \in \h^n$ uniquely determines a Lie algebra element,
and we refer to that element as the \emph{infinitesimal translation} at $p$
in that direction. We view it as an element of the $\mathfrak{g}$-bundle $\widetilde \E$ over $\h^n.$

Given an orthonormal frame $\{ e_i \}$ on an open patch of $\h^n$, we define the corresponding sections $\{ E_i \}$ of the $\mathfrak{g}$-bundle $\widetilde \E$, where $E_i (p)$ is the infinitesimal translation at $p$ in the direction $e_i (p)$.  Define $R_{ij}(p) := \left[ E_i (p) , E_j (p) \right]$ (for $i \neq j$) to be a unit infinitesimal rotation at $p$.  
It represents the skew-symmetric mapping of the tangent space at $p$ that takes $e_j$ to $e_i$, reflecting the fact that $\h^n$ has negative curvature.  It annihilates all vectors orthogonal to 
the plane spanned by $e_i$ and $e_j.$  Define a positive metric $\langle \cdot, \cdot \rangle$ on the bundle $\widetilde \E = \h^n \times \mathfrak{g}$ by taking $\{E_i(p), R_{ij}(p)\}_{i \neq j}$ as an orthonormal basis of the fiber over $p$.  With this metric, the action of $G$ on $\widetilde \E$ is isometric, and 
descends to $\E$ over $X$.
Equipped with this metric on $\E$ 
we can define a pairing on $\Omega^{*}(X;\E)$ by
\[ \langle \alpha, \beta \rangle := \int_X \alpha \wedge {*} \beta, \]
where the coefficients in $\E$ of $\alpha$ and $\beta$ are paired to produce a real number.  Use this pairing to define an adjoint $\delta_\E : \Omega^{k+1}(X;\E) \to \Omega^k (X;\E)$ to $d_\E$ on forms with compact support.  

It is important to note that none of the sections $p \mapsto E_i (p)$ and 
$p \mapsto R_{ij}(p)$ are flat.  
They never represent a locally constant element in $\mathfrak{g}$.  This can be 
seen geometrically from the fact that if the $1$-parameter subgroup determined by 
an infinitesimal rotation at a point $p$ does not fix $q$, then it cannot 
correspond to the same Lie algebra
element as an infinitesimal rotation at $q$.  Similarly, since an infinitesimal
translation at $p$ exponentiates to translation along a geodesic it can only correspond to
infinitesimal translations at other points along that geodesic; the corresponding Lie algebra
elements cannot 
be constant on an open set.

However, it is possible to relate the flat derivative $d_\E$ and its adjoint to the 
covariant derivative (and its adjoint)
determined by the hyperbolic metric on $\h^n.$  The covariant derivative has the advantange 
that it can be understood in local differential geometric terms.  The difference between the 
flat and the geometric derivative is a purely algebraic operator.

It is convenient to describe this relationship in terms of a local orthonormal frame $\{ e_i \}$ of $X$ and the resulting orthonormal basis $\{ E_i, R_{ij} \}$ of $\E$.  Let $\{ \omega_i \}$ be the dual coframe of 
$X$.  Then \cite[Ch.6]{Wu}\cite{MatMur,HK1}
\begin{align*}
d_\E  (V \otimes \alpha) & = \sum_j \left( \omega_j \wedge \nabla_{e_j} (V \otimes \alpha ) +
    \omega_j \wedge ( [E_j,V] \otimes \alpha ) \right)\\
\delta_\E  (V \otimes \alpha) & = - \sum_j i(e_j) \left( \nabla_{e_j} ( V \otimes \alpha) - [E_j,V] \otimes \alpha \right).
\end{align*}

To make sense of these equations we need to explain the term $\nabla_{e_j} (V \otimes \alpha)$, where $V$ is a local section of $\E$ and $\alpha$ is a differential form on $X$.
It equals $(\nabla_{e_j} V) \otimes \alpha + V \otimes \nabla_{e_j} \alpha,$
in other words the connection acts on both $V$ and $\alpha$. Here $\nabla_{e_j} \alpha$ 
denotes the covariant derivative with respect to the Levi-Civita connection associated 
to the hyperbolic metric.  Similarly, the connection $\nabla$ is defined on $\E$ using 
the Levi-Civita connection as follows.  First, $\nabla_{e_j} E_k$ is the Lie algebra element corresponding to $\nabla_{e_j} e_k$.  (Compute $\nabla_{e_j} e_k$ and then capitalize all the $e$'s.)  
Next, $\nabla_{e_k} R_{ij}$ is defined by considering $R_{ij}$ 
as a section of $\so(TX)$, the bundle of skew symmetric endomorphisms of $TX$, and using
the Levi-Civita connection to differentiate this section.
The result will again be skew symmetric, and one converts it back into a linear combination of infinitesimal rotations.
In short, if one writes a local section $V$ as a linear combination of the $\{E_i\}$ and 
of the $\{R_{ij}\}$, it can be viewed as a pair $(u,\widetilde u)$ consisting of a 
section $u$ of $TX$ and $\widetilde u$ of $\so(TX)$, and 
$\nabla V = (\nabla u, \nabla \widetilde u)$
where the Levi-Civita connection is used on these bundles. 
Finally, $i(e_j)$ indicates left 
contraction of the resulting form along the vector $e_j$ 
(eg. $i(v) \alpha \wedge \beta = \alpha(v) \cdot \beta - \beta(v) \cdot \alpha$).

It is useful to express this decomposition as $d_\E  = D + T$ and $\delta_\E  = D^{*} + T^{*}$, where
\begin{equation} \label{DTeqns} \begin{split} 
D & = \sum_j \omega_j \wedge \nabla_{e_j} \\
T & = \sum_j \omega_j \wedge \ad (E_j) \\
D^{*} & = - \sum_j i(e_j) \nabla_{e_j} \\
T^{*} & = \sum_j i(e_j) \ad (E_j),
\end{split}\end{equation}
and $\ad (E_j) (V \otimes \alpha) := [E_j,V]\otimes \alpha$.  Note that $D$ and $D^{*}$ take 
infinitesimal translations (resp. rotations) to infinitesimal translations (resp. rotations), while 
$T$ and $T^{*}$ take infinitesimal translations (resp. rotations) to infinitesimal rotations (resp. translations).

In particular, if we write a local section $s$ of $\E$ as a pair $(u,\widetilde u)$
consisting of a local section of $TX$ (i.e., a vector field) and of $\so(TX)$, 
then both $Du$ and $T \tilde u$ are $1$-forms with values in the infinitesimal translations 
of $X$ and so can be 
thought of as local sections of $TX \otimes T^{*} X$.  Using the identification of 
$TX \otimes T^{*} X$ with $\Hom(TX,TX)$, $T \widetilde u$ is skew-symmetric and corresponds to $-\widetilde u$.  Of particular importance are sections of $\E$ of where 
$T \widetilde u$ cancels out the skew-symmetric part of $Du$.

\begin{defn} \label{defn:canon_lift}
A local section $s = (u, \widetilde u)$ is a \emph{canonical lift} if the skew symmetric part of $D u$ equals $- T \widetilde u$.  A canonical lift is clearly determined by its vector field part $u$, and will be called the canonical lift of $u$.
\end{defn}

To help the reader digest this concept, we give two equivalent definitions.  First, 
the local section $s = (u, \widetilde u)$ is a canonical lift if and only if the skew symmetric part of $D u$ equals $\widetilde u$ viewed as a local section of $\so(TX)$.  Second, if we write 
\begin{equation}\label{eqn:canon_lift}
 d_\E s = \left( \sum b_{ij} E_i \otimes \omega_j \right) +
    \left( \sum c_{k \ell m} R_{k \ell} \otimes \omega_m \right).
\end{equation}
then $s$ is a canonical lift if and only if $b_{ij} = b_{ji}$.  This implies a locally constant section is a canonical lift, because its differential is zero.  We leave the proof of these equivalences to the reader.

On a contractible neighborhood a closed $1$-form $\omega$ (with coefficients in $\E$) has a section $s$ satisfying $d_\E s = \omega$.  This section is unique up to adding a locally constant section, implying that being a canonical lift is a property of the $1$-form $\omega$.  If, on any contractible neighborhood, there is always a canonical lift $s$ satisfying $d_\E s = \omega$, then we will say that $\omega$ locally admits a canonical lift.  Using equation (\ref{eqn:canon_lift}) it follows that $\omega$ admits a canonical lift if and only if $b_{ij} = b_{ji}$.

\begin{lem}\label{lem:canon_lift}\cite[Prop.2.3(c)]{HK1}
Given (globally defined) closed $1$-form $\omega \in \Omega^1 (X,\E)$, there is a global section $s$ such that $\omega + d_\E s$ locally admits a canonical lift.
\end{lem}

Define the (flat) Laplacian on $\E$ as 
$\lap_\E  = d_\E  \delta_\E  + \delta_\E  d_\E .$  Then
\[ \lap_\E  = (D D^{*} + D^{*} D + T T^{*} + T^{*} T) + (D T^{*} + T D^{*} + D^{*} T + T^{*} D).\]
Conveniently, the second term $(D T^{*} + T D^{*} + D^{*} T + T^{*} D)$ is always zero \cite{MatMur}.  
In particular, the Laplacian preserves the decomposition into infinitesimal rotations and translations.  
Thus we can write \[ \lap_\E  = \lap_D + H,\]
where $\lap_D = D D^{*} + D^{*} D$ and $H= T T^{*} + T^{*} T$.  


The idea of the proof of Theorem \ref{thm:mainFuchs} is to represent every cohomology
class in $H^1(X;\E)$ by an $\E$-valued harmonic $1$-form $\omega$, 
one where $d_\E \omega = 0 = \delta_\E \omega$, hence $\lap_\E = 0$, and then show that such a form must equal $0$.  We will, in fact, not do this for the complete
manifold $X$ with Fuchsian ends, but rather for its convex core which is a
hyperbolic manifold $W$ with totally geodesic boundary.  On $W$ we will represent
each class in $H^1(W; \E)$ by a harmonic form satisfying certain boundary
conditions and show that such a form must equal $0$.  Since $H^1(X;\E) = H^1(W;\E)$,
this will prove Theorem \ref{thm:mainFuchs}.

We now state a key formula, called the \W formula, which will motivate the 
boundary conditions on $W$ that we will choose.  It is obtained via integration by parts.

\begin{prop}\label{prop:byparts}\cite[Prop.1.3]{HK1}
Let $W$ be a compact oriented hyperbolic manifold with boundary and let $\omega \in \Omega^1 (W;\E)$ be a smooth
$\E$-valued $1$-form.  Then
\begin{equation} \label{Weitzformula} 
(d_\E \omega, d_\E \omega) + (\delta_\E  \omega, \delta_\E  \omega) = (D \omega, D \omega) + (D^{*} \omega, D^{*} \omega) + (H \omega, \omega) + B,
\end{equation}
where $B$ denotes the boundary term
\begin{equation} \label{Weitzbdry} 
B = - \int_{\partial W} \left( {*}T\omega \wedge \omega + T^{*} \omega \wedge {*} \omega \right) .
\end{equation}
\end{prop}

If $\omega$ is harmonic, which we always take to mean that it is closed and co-closed 
($d_\E  \omega = 0 = \delta_\E  \omega$), the left hand side of the formula (\ref{Weitzformula})
equals $0.$  
When $W$ is at least $3$ dimensional, Weil proved the existence of a $c>0$ such that
\[ ( H \omega, \omega) \ge c \, (\omega,\omega) \] 
for all $1$-forms $\omega \in \Omega^1 (W;\E)$ \cite{We}.  In the case when $W$ is closed, the boundary term $B$ 
is trivial and this formula implies that any harmonic $\E$-valued $1$-form must be $0.$  By the
Hodge theorem, this means that $H^1(W;\E) = 0$ which proves infinitesimal (hence local)
rigidity in the closed case for these dimensions.

Similarly, when $W$ has boundary, if one can show that any class in $H^1(W;\E)$ has a
harmonic representative where the the boundary term $B$ of (\ref{Weitzbdry}) 
is non-negative, this will again
prove that $H^1(W;\E) = 0.$  This is the basis of the proof of our main theorem, 
which we now state in the form it will be proved.


\begin{thm}\label{thm:main}
Let $W$ be a compact hyperbolic $n$-manifold with totally geodesic boundary.  
If $n>3$ then $H^1(W;\E) = 0$ so that $W$ is infinitesimally rigid.  Hence,
$W$ is locally rigid.
\end{thm}

\section{Harmonic forms on the boundary}\label{sec:bdryforms}


In this section we will begin the proof of Theorem \ref{thm:main}.

Let $W$ be a compact hyperbolic manifold with totally geodesic boundary and dimension 
$n \geq 4$.  Equivalently, let $X$ be a complete hyperbolic $n$-manifold which has Fuchsian ends and is not Fuchsian, and let $W$ be its convex core.  Define $M := \partial W$.  The goal is to prove $H^1(W;\E) \cong H^1(X;\E) = 0$.  The idea is to consider an infinitesimal deformation of the hyperbolic structure of $W$, 
described as a deRham cohomology class in 
$H^1(W;\E)$.  We find a harmonic representative $\omega$ of this cohomology class which is an $\E$-valued $1$-form on $W$ satisfying
certain boundary conditions.  We show that these boundary conditions can be chosen so 
that the boundary term $B$ in the \W formula from Proposition \ref{prop:byparts} will equal $0$.  
This implies that $\omega = 0$ and hence that $H^1(W;\E)$ is trivial.

There are three main steps in this proof.  Each has its own section. 
First, we analyze the structure of $\E$-valued harmonic forms on the boundary $M$  and show
how to extend them to model harmonic forms 
in a neighborhood of the boundary.  Any class 
in $H^1(W;\E)$ is represented by an $\E$-valued $1$-form ${\hat \omega}$ which is closed
and has this model harmonic 
structure near the boundary.  Thus, $d_\E  {\hat \omega} = 0$ on all of $W$, and 
$\delta_\E  {\hat \omega}$ is $0$ near the boundary, but $\delta_\E {\hat \omega}$ is not necessarily zero on all of $W$. 
To find a globally harmonic representative we must solve the problem of finding 
a section $s$ of $\E$ satisfying 
$\delta_\E  d_\E  s = - \delta_\E  {\hat \omega}$ on $W$.  Then 
$\omega = {\hat \omega} + d_\E  s$ is a harmonic form in the same cohomology class.  
In the second part we describe a boundary value problem for finding such a section $s$, 
and show that it is uniquely solvable.  In the final section we compute the boundary
term $B$ in the \W formula (\ref{Weitzformula}) and show that, for the harmonic form $\omega$ constructed
from our boundary value problem, this term is trivial.  This implies that $\omega$
itself must be trivial.

Because much of the analysis takes place on the boundary $M$ of 
the manifold $W$, we will denote the dimension of $M$ by its own letter, $m$.  
The dimension of $W$ will be then be denoted by $n = m+1$.  
Abusing notation, it will be convenient to let the letter $n$ also denote a vector normal to $M$. 

\bigskip
 
As sketched above, in this section we will describe harmonic $\E$-valued $1$-forms on, and in a 
neighborhood of the boundary $M$ of $W$.  
For this we choose local frames and coframes that reflect the geometry of the situation. 
At any point in $M$, choose a local orthonormal frame $\{n,e_1,\ldots,e_m\}$, where $n$ is an outward
normal vector, and the $e_i$ are tangent to $M$.
Let $\{N,E_1,\ldots,E_m\}$ denote the corresponding infinitesimal
translations, which are local sections of $\E$.  There are $m$ rotational generators $R_{ni}$, $i =1, \ldots, m$ corresponding to infinitesimal rotations in planes perpendicular to $M$.
They will play a special role in what follows.

As 
$W^{n}$ has totally geodesic boundary $M^m$, the 
restriction to a connected component of $M$ induces a representation of its fundamental group
into the orientation preserving isometries of $\h^m$ viewed as a subgroup of the isometries
of $\h^{n}$.  
For simplicity, we will discuss the case when the boundary is 
connected.  In general, the following 
argument goes through component by component,  
identifying a corresponding invariant hyperplane in the universal cover 
with $\h^m$ in each case.  

The adjoint action of the subgroup preserving $\h^m$ preserves a direct sum 
decomposition $\mathfrak{g} = \mathfrak{h} \oplus \mathfrak{s}$ of the Lie algebra
$\mathfrak{g}$ of $G = \Isom (\h^n)$. 
Here $\mathfrak{h} \cong \so (1,m)$
is the Lie algebra of the subgroup preserving $\h^m.$  The other factor
$\mathfrak{s}$ does not 
come from a subalgebra of $\mathfrak{g}$ but
is preserved under the adjoint action of the isometries of $\h^m$.
The infinitesimal version of this statement is that 
$[\mathfrak{h}, \mathfrak{s}] \subset \mathfrak{s}$.  We will describe this
factor in more detail below.

This decomposition of $\mathfrak{g}$ induces a decomposition of 
the bundle $\widetilde \E$ restricted to $\h^m.$  It is an orthogonal 
decomposition with respect to the metric on this bundle. 
Because it is invariant under the adjoint action, it descends to a decomposition 
of $\E$ over $M$; we denote the corresponding 
sub-bundles by $\mathcal{H}$ and $\mathcal{S}.$   Furthermore, it induces
a direct sum decomposition of the cohomology group $H^1(M;\E).$  
We write these decompositions
as $\E  \cong \mathcal{H} \oplus \mathcal{S}$ and 
$H^1(M;\E) \cong H^1(M;\mathcal{H}) \oplus H^1(M;\mathcal{S})$ .  

The cohomology group $H^1(M;\mathcal{H})$ parametrizes the infinitesimal
deformations of the $m$-dimensional hyperbolic structure on $M$.
In terms of the basis 
above, 
$\mathcal{H}$ is generated by $\{E_i\}$ and $\{R_{jk}\}$ for $i, j, k \in \{ 1,2,\ldots, m\}$. 
In other words, there is no normal component.

As a vector space, 
the factor $\mathcal{S}$ is generated by $N$ and $\{R_{ni}\}$.
This factor measures how 
the hyperplane $\h^m$ is being moved away from itself.  At any point in $\h^m$, the element
$N$ corresponds to infinitesimal translation in the direction
normal to the hyperplane.  An element $R_{ni}$ rotates in the plane
spanned by $n$ and $e_i$, rotating the tangent vector $e_i$ in the
hyperplane to 
the normal vector $n$.  The subspace in $\h^m$
orthogonal to $e_i$ is fixed under this rotation.  More intrinsically,
we can consider the space of all hyperplanes in $\h^{m+1}$, 
called $(m+1)$-dimensional deSitter space.  Then the fibers 
$\mathfrak{s}$ of $\mathcal{S}$ can be identified with the tangent space at our particular
hyperplane of deSitter space.

With respect to the positive definite 
metric on $\mathcal{S}$ induced as a subspace of 
$\E$, $N$ is orthogonal
to the 
$R_{ni}$.  Thus, $\mathcal{S}$ splits (as a metric bundle)
into a direct sum $\mathcal{S} \cong N \oplus \mathcal{B}$ (``$\mathcal{B}$"
for ``bending").  This is a reflection of the following useful
observation.  We can identify each infinitesimal rotation $R_{ni}$
with the corresponding infinitesimal translation $E_i$.  Extending 
this identification linearly, we can identify the subspace generated
by these rotations with the tangent space of 
$\h^m$.  The subspace
$N$ corresponds to the normal bundle of 
 $\h^m$ inside $\h^{n}$.
These identifications descend to $M$ and the bundle $\mathcal{S}$
can be identified with the tangent bundle of $M$ plus a trivial
line bundle.  Note that this decomposition of the bundle is not
compatible with the local flat structure; in particular, it does not
induce a decomposition of $H^1(M;\mathcal{S})$.

Now given an equivalence class $[\hat \omega] \in H^1(W;\E)$ of
infinitesimal deformations of the hyperbolic structure on $W$,
we can restrict it to the boundary $M$ and consider the
resulting class 
$[{\hat \omega}_M] \in H^1(M;\mathcal{H}) \oplus H^1(M;\mathcal{S}).$
When $M$ is compact and has dimension at least $3$, the first factor is
always trivial.  We assume that we are in this case.  (The case when $M$ has
dimension $2$ will be discussed briefly in Section \ref{sec:conjectures}.) 

Consider $[{\hat \omega}_M] \in  H^1(M;\mathcal{S}).$
Even when $M$ is compact and $m\geq 3$, it can be nontrivial.  Our
first step is to describe the 
harmonic elements in $ H^1(M;\mathcal{S})$.

\begin{thm} \label{harmonicbends}
Let $M$ be a closed hyperbolic $m$-manifold, $m\geq 2.$
Then $ H^1(M;\mathcal{S})$ is isomorphic to
the space of $TM$-valued $1$-forms $B$ on $M$ satisfying the equations 
$D B = 0 = D^{*} B$ and so that, when viewed as elements of
$\Hom(TM,TM)$, they are symmetric and traceless.
\end{thm}

In this theorem the operators $D$, $T$, $D^*$, $T^*$, $d_\E$, and $\delta_\E$ are understood to live on $M$, so, in particular, the covariant derivatives and bracket operations in the formulae (\ref{DTeqns}) for these operators are only in directions tangent to $\h^m$.
Note that this theorem holds even
in the case $m=2$, where 
the $1$-form $B$ corresponds to a 
holomorphic quadratic differential, and we can view the general case as being some kind of 
higher dimensional infinitesimal Schwarzian derivative.

\begin{proof}
The Hodge Theorem says that we can represent the cohomology class 
$[{\hat \omega}_M]$ by a unique
harmonic element $B$. 
Harmonic here means
closed and co-closed with respect to the flat structure on $\mathcal{S}$ so that 
$d_\E B = 0 = \delta_\E  B$.
We can decompose $d_\E$ and $\delta_\E$  into $D + T$ and $D^{*} + T^{*}$.

The laplacian $\lap_\E  = d_\E  \delta_\E  + \delta_\E  d_\E$ decomposes into
$\lap_D + H$ where $\lap_D = D^{*} D + D D^{*}$ and $H = T^{*} T + T T^{*}$.
Via 
integration by parts, both $\lap_D$ and $H$ are non-negative operators.  We
will now identify the kernel of $H$.
To do this, note that $T$ and $T^{*}$ switch the two factors $N$ and $\mathcal{B}$
in the orthogonal decomposition of $\mathcal{S}$ and hence that
$T^{*} T$ and $T T^{*}$
preserve this decomposition.  It is straightforward to compute that on $N$ the operator 
$T^{*} T$ is multiplication by $m-1$ and $T T^{*}$ is multiplication by $1$.  So the sum is
multiplication by $m$.  Viewing an element of $\Omega^1(M;\mathcal{B}) \cong \Omega^1(M;TM)$ 
as an element of $\Hom(TM,TM)$, one computes that $T T^{*}$ annihilates the traceless part and is 
multiplication by $m$ on multiples of the identity transformation.  Also, $T^{*} T$ is multiplication by $2$ on 
the skew symmetric part and $0$ on the symmetric part.  Thus, the sum of the two operators is 
multiplication by $m$ on multiples of the identity, 
multiplication by $2$ on the 
skew symmetric part
and $0$ on the traceless, symmetric
part.  (This can also be computed directly using 
formula (4) in \cite{HK1}.)
We conclude that $H$ is positive semi-definite with kernel equal to the symmetric, traceless 
elements of $\Omega^1(M;\mathcal{B})$.  Note that such elements are, in fact, in 
the kernels of both $T$ and $T^{*}$ individually.  It follows that the harmonic elements 
$B$ of $H^1(M;\mathcal{S})$ satisfy $T B = 0 = T^{*} B$ and $DB = 0 = D^{*}B$.
\end{proof}

Now, given a harmonic element $B \in \Omega^1(M;\mathcal{S})$, we extend it to a 
harmonic element in $\Omega^1(M\times \R;\E)$  where $M \times \R$ is equipped with the
Fuchsian hyperbolic structure coming from the inclusion $\h^m \subset \h^n$. 
This extension can be defined by simply 
pulling back $B$ 
via the orthogonal 
projection map 
$\pi: M \times \R \to M$ onto the totally geodesic copy of $M = M \times \{0 \}$. 
A pull-back of a closed form (using the flat coefficients)
is always closed.  Whether or not a pull-back is co-closed depends on the underlying metrics and the map.

To see that the pull-back is co-closed in our case, we express everything 
in terms of an orthonormal frame. 
Given an orthonomal frame and co-frame $\{e_i\}$ and $\{\omega_i\}$ 
on $M$, the product structure determines tangent vectors and $1$-forms, $\{\pi^{*} e_i\}$ and 
$\{\pi^{*} \omega_i\}$ 
along the slices $M \times \{r\}$.  The corresponding orthonormal frames and coframes with respect
to the Fuchsian hyperbolic metric on $M \times \R$, denoted again by $\{e_i\}$ and $\{\omega_i\}$, equal
$\sech r \{\pi^{*} e_i\}$ and $\cosh r \{\pi^{*} \omega_i\}$.
We denote by $n$ the unit vector orthogonal to the slices $M \times \{r\}$, pointing in the positive direction along $\R$.

For a fixed value of $r$ the points on the hypersurface $M \times \{r\}$ have constant distance 
$|r|$ from the totally geodesic $M \times \{0\}$.  This implies that it is totally umbillic
with constant normal curvature $\tanh r.$  From this we conclude that $\del_{e_i} n = \tanh r \, e_i$ 
and that $(\del_{e_i} e_j) \cdot n = - \tanh r \, \delta_{ij}$ at any point of $M \times \{r\}.$
Furthermore, $\del_n e_i = 0 = \del_n n.$ 
The values $c_{ijk} = \langle \nabla_{e_j} e_i , e_k \rangle$ depend on our choice of frame.
However, when the frame on $M \times \{r\}$ is determined as above
by the frame on $M \times \{0\}$, the value of any $c_{ijk}$ at a point 
$(x,r)$ equals $\sech r$ times its value at $(x,0)$.  In particular, if we choose a frame 
near $(x,0)$ using geodesic coordinates, the $c_{ijk}$ will equal $0$ at $(x,r)$ for
all values of $r$.

With these observations we can now describe our extended form and show that it has the properties we want.

\begin{prop} \label{pullback}
Let $B \in \Omega^1(M;\mathcal{B})$ be closed and co-closed where $M$ is an $m$-dimensional
hyperbolic manifold, $m\geq 2$.  Suppose that
we can write $B = \sum b_{ij} R_{ni} \otimes \omega_j $, where the 
$b_{ij}$ are functions on $M$ which determine a matrix that is symmetric and traceless at each point.  
With respect to the extended orthonormal frame and coframe on $M \times \R$, let
$\widetilde B = \sum b_{ij} R_{ni} \otimes \omega_j$ and
$\widetilde A = -\tanh r \,  \sum b_{ij} E_i \otimes \omega_j$.
Then the $\E$-valued $1$-form $\omega_0 = \widetilde B + \widetilde A$ is 
closed, co-closed, equal to $\pi^{*} B$, and satisfies $D^{*} \omega_0 = 0$. 
\end{prop}

In other words, $\widetilde B$ trivially extends $B$ 
to each slice $M \times \{r\}$ using the extended 
frame and coframe and, under the identification of $\mathcal{B}$ with $TM$,  $\widetilde A$ simply equals 
$-\tanh r \, \widetilde B$. 

\begin{proof}
First, we claim that the $1$-form $\omega_0 = \widetilde B + \widetilde A$ equals the pull-back $\pi^{*} B$.  To see this
let $x\in M$ and $r \in \R$.  Observe that the element $R_{ni}$ of $\E$ at $(x,0)$ 
equals $\cosh r \, R_{ni} - \sinh r \, E_i$ at $(x,r)$.  On the other hand, the $1$-form 
$\omega_j$ at $(x,0)$ pulls back to $\sech r \, \omega_j$ at $(x,r)$.  The 
formula $\pi^{*} B = \widetilde B + \widetilde A$ follows, implying that $\omega_0$ is closed.

To see that it is co-closed refer to the formula (\ref{DTeqns}) for $\delta_\E = T^* + D^*$ on $\E$-valued $1$-forms and see how the value of $\delta_\E \omega_0$ differs at $(x,r)$ from that at $(x,0)$.  
First, consider the algebraic operator $T^{*}$.  When applied to $\widetilde B$ on $M \times \R$, 
the only difference from being applied to $B$ on $M$ is the addition of the term corresponding to
the normal direction.  Since there is no $dr$ term in $\widetilde B$ this adds nothing
and $T^{*} \widetilde B = 0$ since $T^{*} B  = 0$.  Also,
$T^{*} (E_i \otimes \omega_j) =  [E_j,E_i] = R_{ji}$. Since $b_{ij} = b_{ji}$
and $R_{ij} = - R_{ji}$ we obtain $T^{*} \widetilde A  = 0$.

Now consider the differential operator $D^{*}$.  At any point $(x,0)$ we can choose
a frame and coframe based on geodesic coordinates. As noted above this implies that 
$\nabla_{e_i} e_j = -\tanh r \, \delta_{ij} \, n$ for all $i,j $ with respect to the 
extended frame at $(x,r)$ for any $r$.
One computes that  
\[ D^{*}(R_{ni} \otimes \omega_j) = 
-\nabla_{e_j} R_{ni} + (\sum_k \omega_j(\nabla_{e_k} e_k)) \, R_{ni}.\]
Because of our choice of frame the second term vanishes.
Furthermore, 
\[\nabla_{e_j} R_{ni} =  (\tanh r \, R_{ji} + \sum_k c_{ijk} R_{nk})\]
where $c_{ijk} = \langle \nabla_{e_j} e_i , e_k \rangle.$
In our choice of frame only the first term is nontrivial.  Thus, using geodesic 
coordinates, we obtain
\[ D^{*}(b_{ij} (R_{ni} \otimes \omega_j)) = - ( b_{ij} \tanh r \, R_{ji} + 
(\nabla_{e_j} b_{ij}) R_{ni}).\] 
Summing over $i,j \in {1,2,\cdots, m}$, the contribution from the terms involving $R_{ji}$ equals $0$ because $b_{ij} = b_{ji}$ and $R_{ij} = -R_{ji}$.
Since the functions $b_{ij}$ are independent of $r$ and the orthonormal
frame at $(x,r)$ equals $\sech r$ times the frame pulled-back from $(x,0)$, 
the terms involving derivatives of $b_{ij}$ at $(x,r)$ equal $\sech r$ times 
those at $(x,0)$.  
Since $D^{*} B = 0$ on $M$ we conclude that $D^{*} \widetilde B = 0$ on $M \times \R$.

Similarly, to compute $D^{*} \widetilde A$ we observe that
\[ D^{*}(E_i \otimes \omega_j) = -\nabla_{e_j} E_i 
+ (\sum_k \omega_j(\nabla_{e_k} e_k)) \, E_i.\]
In geodesic coordinates the second term is trivial and the first term
is nontrivial only when $i=j$ in which case it equals $\tanh r \,  N.$
We obtain
\[D^* (\tanh r \, b_{ij} \, (E_i \otimes \omega_j)) = 
(\delta_{ij} b_{ij}) (\tanh r)^2 \, N - (\nabla_{e_j} b_{ij})\, \tanh r \, E_i\] 
Summing over $i,j$, the terms involving $N$ only occur when
$i = j$ so the contribution from them is zero because $B$ has trace zero. 
The contribution from the terms involving derivatives of the
$b_{ij}$ equals zero as before.

\end{proof}

We will refer to $\omega_0$ as the model harmonic form near the boundary $M$ of $W$. 


\section{Boundary value problem} \label{sec:bdryvalue}

In this section we will describe and solve a boundary value problem that will provide us
with a harmonic representative $\omega$ for any class in $H^1(W;\E).$  The boundary conditions
have been chosen so that the boundary term (\ref{Weitzbdry}) in the \W formula is reasonably computable.
In the next section we will show that the boundary term of (\ref{Weitzbdry}) is $0$, which will imply that $\omega$ is trivial. 

Given an element $[\hat \omega] \in H^1(W;\E)$ where $W$ is a 
hyperbolic $(m+1)$-dimensional manifold 
with totally geodesic boundary $M^m$, we know that, when $m \ge 3$, 
the restriction $[{\hat \omega}_M] \in H^1(M;\E)$
lies completely in $H^1(M; \mathcal{S})$.  By 
the results of the previous section, there is a unique
harmonic $B \in \Omega^1(M;\mathcal{B})$ representing this class.  Under the identification
of $\mathcal{B}$ with $TM$, and $\Omega^1(M;TM)$ with $\Hom(TM,TM)$, the linear transformation
corresponding to $B$  is symmetric and traceless.   
By Proposition \ref{pullback}, $B$ determines a model harmonic form $\omega_0$ defined in a neighborhood of $M$. 
The element $[{\hat \omega}] \in H^1(W;\E)$ 
will be cohomologous to $\omega_0$ 
in a neighborhood of the boundary of $W$.
We can therefore 
assume that $[{\hat \omega}]$ is represented by a form ${\hat \omega}$
which 
agrees with $\omega_0$ 
in a neighborhood of the boundary $M$.
In particular, it is closed and co-closed near $M$.

However, ${\hat \omega}$ will not necessarily 
be harmonic on all of $W$.  It is closed, 
but $\delta_\E  {\hat \omega}$ will generally not be trivial
away from the boundary.  To find a harmonic representative we need to find a section
$s$ of the bundle $\E$ over $W$ satisfying
\begin{equation}\label{correctionprob}
\delta_\E  d_\E  s = -\delta_\E  {\hat \omega}.
\end{equation}
Then $\omega = {\hat \omega} + d_\E  s$ will be a harmonic
representative of our given cohomology class.

There are many solutions to (\ref{correctionprob}). We 
will require $s$ to satisfy
certain boundary conditions that will make the operator $\lap_\E  = \delta_\E  d_\E$ self-adjoint 
and elliptic with trivial kernel.  Since we want to apply the \W formula (\ref{Weitzformula}) to this
harmonic representative, our choice of boundary conditions will also be motivated
by the resulting boundary term in that formula.  Recall that the 
boundary term (\ref{Weitzbdry})
is an integral over $M$ with integrand 
$-({*}T \omega \wedge \omega + T^{*} \omega \wedge {*}\omega).$  In this section we
will show that our chosen boundary conditions imply that
$T^{*} \omega = 0$, greatly simplifying 
our computation of the boundary integral in the next section.

By Lemma \ref{lem:canon_lift} we can assume ${\hat \omega}$ admits a canonical lift, meaning it is locally the image under $d_\E$ of a canonical lift.
Equivalently, we can assume the translational part of ${\hat \omega}$,
viewed as a element of $\Hom(TW,TW),$ is symmetric.

The operator $\lap_\E  = \delta_\E  d_\E$ acting on sections preserves the decomposition 
of $\E$ into translational and rotational parts.  It follows that
solving equation (\ref{correctionprob}) is equivalent to solving
it separately for the translational and rotational parts; in other words, to solving the corresponding
problem for vector fields and for sections of $\so(TW).$
Using canonical lifts, 
we will see below that it suffices to solve the vector field problem.
In particular, when solving (\ref{correctionprob}) we will 
be able to assume that the global section $s$ is a canonical lift.  This 
will imply that the translational part of $d_\E s$ will be symmetric.  Thus, the harmonic representative 
$\omega = {\hat \omega} + d_\E s$ will also have this property.

Using the metric on $TW$ we can identify vector fields with $1$-forms and
linear transformations with $(0,2)$-tensors. 
We will use the convention that a linear transformation $\phi$ is identified with the $(0,2)$-tensor $\theta$ defined by
\[\theta (u_1, u_2) := \langle \phi u_1, u_2 \rangle.\]
Notice that a skew symmetric transformation is identified with a $2$-form.  
For example, with this convention $R_{ij}$ is identified with $\omega_j \wedge \omega_i$.
A section of $\E$ can then be identified 
with a pair $(\tau, \widetilde \tau)$ consisting of a $1$-form and a $2$-form. 
Via 
this identification
a canonical lift is identified with the pair 
$(\tau, \frac {1}{2} d \tau)$, 
where $\tau$ is dual to $u$ and $d$ is exterior differentiation on $W$.  
(In low dimensions the skew symmetric part of $Du$ is identified with 
$\frac {1}{2} \curl(u).$)

On the translational part of $\E$, the operator $T^{*}T$ is multiplication by $m$, so
$\lap_\E  u = D^{*}D u + T^{*}T u = D^{*} D u + m u$.  Similarly, on the rotational part of $\E$, 
$T^{*}T$ is multiplication by $2$ so that $\lap_\E  \widetilde u = D^{*}D \widetilde u + 2 \widetilde u.$ 
If we view a section of $\E$ as a pair consisting of a vector field and a section of $\so(TW)$, 
the operator $D^{*}D$ equals $\nabla^{*} \nabla$ on each component of the pair.
Under the identification of vector fields with $1$-forms and sections of $\so(TW)$ with $2$-forms,  
$\nabla^{*} \nabla$ becomes $\nabla^{*} \nabla$ on these forms.  
The fact that the Ricci curvature of $\h^{n}$
equals $m$ implies that $\nabla^{*} \nabla$ on $1$-forms equals 
$ \lap + m$ where $ \lap =  \delta  d +  d  \delta$
is the usual (exterior) laplacian on differential forms.  
More generally, for an $n$-dimensional space of constant
curvature $-1$, $\nabla^{*}\nabla$ on $p$-forms equals 
$ \lap + p (n-p)$ \cite[Ch.7.4]{Pet}. 
When $p=2$ and $n = m+1$, $p (n-p) = 2 (m-1) = 2m-2.$
Thus, under the identification of $u$ and $\widetilde u$ with the forms
$\tau$ and $\widetilde \tau$, the operator $\lap_\E$ corresponds to
$ \lap + 2m$ in both cases.

Being a canonical lift means that the
$2$-form $\widetilde \tau = \frac {1}{2} d \tau.$  
Since $d$ commutes with $\lap + 2m$, it follows that $\lap_\E$ 
preserves the property of being a canonical lift.
We are assuming that ${\hat \omega}$ is the local image under $d_\E$ of a canonical lift.  
It follows that $\delta_\E {\hat \omega}$ is a canonical lift. 
If $\zeta$ denotes the $1$-form dual to the translational part of 
$ - \delta_\E {\hat \omega}$, then the $2$-form corresponding to its rotational part 
equals $\frac{1}{2}  d \zeta$.  We conclude that in order to solve equation (\ref{correctionprob}), 
it suffices to solve
\begin{equation}\label{formprob} 
(\lap + 2m) (\tau) = \zeta
\end{equation}
for a $1$-form $\tau$ on $W$. The solution  to (\ref{correctionprob}) will then be the
section $s$ which is the canonical lift of the vector field $u$ dual to $\tau.$

We will now describe our boundary conditions both in terms of the translational
part $u$ of the section $s$ and the $1$-form $\tau$ dual to $u$.  At any point
on the boundary $M$ of $W$ we can decompose vector fields and $1$-forms into
their normal and tangential components. Using surfaces equidistant from the boundary and 
orthogonal projection, this decomposition extends to a neighborhood of the boundary.  
Choose the unit normal $n$ to be pointing outward and the normal coordinate to be denoted by $r$, 
so that $dr$ is dual to $n$.  Near the boundary write $u = (h,v)$, where $h$ is a function 
and $v$ is a vector field on $M$.  The dual $1$-form is $\tau = h dr + \sigma$,
where $\sigma$ is dual to $v$ on $M$.

Our boundary conditions are
\begin{equation}\label{vfbdcond1} 
h = 0 
\end{equation}
\begin{equation}\label{vfbdcond2} 
\del_n v = 0.
\end{equation}
In other words, the normal component of $u$ is zero and the normal derivative
of the tangential part of $u$ is zero.  

We 
must show that solving (\ref{correctionprob}) subject to 
boundary conditions (\ref{vfbdcond1}) and (\ref{vfbdcond2}) is an elliptic problem
with a unique solution.  Furthermore, we need to show that the resulting harmonic
$\E$-valued $1$-form $\omega = {\hat \omega} + d_\E  s$ on $W$ 
satisfies $T^{*} \omega = 0$.

Let $s$ be a global section of $\E$ and write its decomposition
into translational 
and rotational parts as $s = (u,\widetilde u)$.  Decompose
$u = (h,v)$ into its normal and tangential components.  
Let $\tau$ be the $1$-form dual to $u$ and $\widetilde \tau$ 
the $2$-form associated to $\widetilde u$ (which is a section of $\so(TW)$).  
Assume that 
$s$ is a canonical lift so that $2 \widetilde \tau = d \tau$.  
Decompose 
$\tau$ and $\widetilde \tau$ into their normal and tangential parts as $\tau = h dr + \sigma$
and $\widetilde \tau = dr \wedge \widetilde \sigma + \widetilde \phi$.  Here $\sigma$ and $\widetilde \sigma$
are tangential $1$-forms and $\widetilde \phi$ is a tangential $2$-form.  Since we
are assuming that
$2 \widetilde \tau = d \tau$, we obtain
$2 \widetilde \sigma = \del_n \sigma - d_M h$, 
where $d_M$ is the exterior derivative in the tangential directions.  

Using 
canonical lifts, our boundary problem (\ref{correctionprob}) reduces to the analogous problem for the
vector field part of the section $s$, subject to 
boundary conditions 
(\ref{vfbdcond1}) and (\ref{vfbdcond2}).  With the above notation, the boundary conditions 
for the dual problem (\ref{formprob}) on $1$-forms are:

\begin{equation}\label{formbdcond1} 
h = 0 
\end{equation}
\begin{equation}\label{formbdcond2} 
\del_n \sigma = 0.
\end{equation}
Using equation (\ref{formbdcond1}) we have that
$d_M h \equiv 0$ on the boundary, so
we could equivalently replace the second condition with

\begin{equation}\label{formbdcond3} 
\widetilde \sigma = 0. 
\end{equation}

In the formalism of differential operators on the bundle of $\mathcal{C}^{\infty}$ $1$-forms
we 
write the boundary value problem as solving $P(\tau) = -\zeta$ subject
to the conditions $p_1(\tau) = 0 = p_2(\tau)$ where $(P;\{p_1,p_2\})$ is the operator
\begin{align*}
P(\tau) &= (\lap +2m) (\tau)\\
p_1(\tau) &= h\\
p_2(\tau) &=   \del_n \sigma. 
\end{align*}
To show that this boundary problem has a unique solution, it suffices to show that this
differential operator is elliptic and that, on the subspace
where $p_1(\tau) = 0 = p_2(\tau)$, it is self-adjoint and positive.

These are standard facts because the main operator $P$ has the same symbol as the laplacian
on $1$-forms, and because, near the boundary, the operator decomposes into the 
normal part, where the boundary condition is the Dirichlet condition, and the tangential part, where the boundary condition is the Neumann condition.  However, for completeness, we include the argument here. 

We will first show that this operator is self-adjoint and positive.
Recall that
$\lap =  d  \delta +  \delta  d$.  Then
for real-valued $1$-forms $\tau$ and $\psi$ 
on $W$, integration by parts yields: 
\[ \langle  \lap \tau, \psi\rangle = \langle  d \tau,  d \psi \rangle ~+~
\langle  \delta \tau,  \delta \psi \rangle + \beta(\tau,\psi), \]
where $\langle \cdot, \cdot \rangle$ is the $L^2$ inner product on $W$ and
$\beta(\tau,\psi)$ is a term given by an integral over the boundary. 

The operator $\lap$ (hence $\lap + 2m$) will be self-adjoint as long as
\[ \langle  \lap \tau, \psi\rangle ~-~ \langle \lap \psi, \tau\rangle
~=~ \beta(\tau,\psi)\,-\, \beta(\psi, \tau) ~=~ 0. \]
The operator $\lap + 2m$ will have trivial kernel as long as
$\langle (\lap +2m) \tau, \tau \rangle > 0$
 for any nonzero 
$\tau$.  Letting $\tau = \psi$ above,  
this will be guaranteed as long as we have
$\beta(\tau,\tau) \geq 0$.

Using Green's identity, we obtain the following formula for the boundary term, where
the boundary is oriented with respect to the outward normal:
\[ \beta(\tau, \psi) ~=~  \int_M ~  {*} d\tau \wedge \psi ~+~  \delta \tau \wedge {*} \psi.\]
As before we decompose $\tau$ as $\tau = h \, dr + \sigma$ and  
$\psi$ as $\psi = k \,dr + \phi$.  Boundary condition (\ref{formbdcond1}) implies that
$h=0=k$, so that the
tangential parts of $ {*} \tau$ and $ {*} \psi$ are trivial.  
Similarly, boundary condition (\ref{formbdcond3}) implies that the normal parts of 
$d \tau$ and $d \phi$, hence the tangential parts of 
$ {*}  d \tau$ and $ {*}  d \phi$, are trivial.  It is then clear that
$\beta (\tau, \psi) = 0$ for $\tau$ and $\psi$ satisfying the boundary conditions and the operator
is thus self-adjoint and positive.

To see that the boundary conditions lead to an
elliptic boundary value problem, we consider the operators,
$P$, $p_1$, and $p_2$, 
where the ranges of the two boundary operators are
the sub-bundles of normal and tangential parts, respectively, of $1$-forms on the boundary.
We then take the top order terms of each of these operators.  It is a subtlety
of differential operators on bundles
that ellipticity may depend on the choice of decomposition of the target
bundle, since this affects what the top order terms are.
In our case, the boundary operators themselves decompose 
(ie., you can decompose the domain in the same way).

To show that the system is elliptic, one considers the symbols of the operators.
This amounts to looking at the system in local coordinates, fixing the coefficients
of the operators by evaluating at a boundary point, and taking only the top
order terms in each operator.  One then considers the homogeneous, constant
coefficient problem in the upper half space of $\R^n$ given by these simplified
operators.  

The symbol of $\lap + 2m$ is simply the symbol of the  
standard laplacian on $\R^n$, which is well-known to be elliptic.
Clearly $p_1$ is of order $0$ and $p_2$ is of order $1$.  Taking the top order terms,
the operators become the same operators viewed in the upper half-space
$\R^n_+ = \{(x_1, \ldots, x_m,t) \, | \, t \geq 0\}$ with the standard Euclidean metric.
We are now left with the simplified system of solving
$\lap \tau = 0$ in the upper half-space with homogeneous boundary
conditions determined by these Euclidean boundary operators.
By definition,
the original system is elliptic if and only if
this simplified  system has no nontrivial bounded solutions.
It suffices to show that there are no nontrivial bounded
solutions $\tau(x,t)$ of the form
$f (t) e^{i(\zeta \cdot x)},$ where $x = (x_1,\ldots, x_m)$ and $\zeta$ is any 
nonzero vector in the boundary hyperplane.  The solutions $f (t)$ 
to $\lap \tau = 0$ for a given choice
of $\zeta$ are linear combinations of $e^{|\zeta| \, t}$ and $e^{-|\zeta| \,t}$.  
Since we are interested in bounded solutions, only scalar multiples of the latter function
will appear.  In particular, we have that 
$f' (t) = -|\zeta| f (t)$. 

We decompose $f$ into its normal and tangential components, which
we denote by $h \, dr$ and $\sigma$, respectively.  The boundary conditions are then 
$h(0) = 0$ and $\sigma^{\prime}(0) = 0$.  But since $\sigma ^{\prime}(t) = -|\zeta| \sigma(t)$,
the latter condition implies that $\sigma(0) = 0$ and we conclude that any solution must
be trivial.

Thus we have shown that our boundary value problem is uniquely solvable.  Standard elliptic
theory implies that the solution will be smooth, even on the boundary.  Having solved the 
equation $\delta_\E  d_\E  s = - \delta_\E  {\hat \omega}$ we define
$\omega = {\hat \omega} + d_\E s.$   Then $\omega$ is a closed and co-closed representative in 
the same comology class as ${\hat \omega}$.  We will now
show that $T^{*} \omega = 0.$

Define the trace of $\omega$ as the trace of its translational part, viewed as an element of $\Hom(TW,TW)$.  We denote this trace function on $W$ by $\tr$.  To prove $T^{*} \omega = 0$ it will suffice to show that $\omega$ has trace zero. 
Since $\omega$ is closed, it is the image under $d_\E$ of a locally defined section
of $\E$.  The trace is just the divergence of the locally defined vector field
which is the translational part of this locally defined section.  It also equals 
$-\delta$ of the local $1$-form which is dual to this local vector field.  
Note that, although this vector field and its dual $1$-form are only locally defined, 
the trace is, nonetheless, globally defined.

Our proof that $\omega$ has trace zero is similar to an argument of \cite{HK2}. 
We know that 
$\omega$ is harmonic.  Since $\omega$ is closed,
it is locally equal to $d_\E$ of a local section, and being co-closed as
well means that $\lap_\E$ of this local section is zero. 
$\lap_\E$ preserves the decomposition into translational and rotational parts, so $\lap_\E$ of the corresponding 
local vector field is zero. 
As we showed above, this means that 
$  d  \delta +  \delta  d + 2m$ applied to the dual local $1$-form is zero. 
Applying $ \delta$ and using that $\tr$ equals $-\delta$ of this 
local $1$-form, we conclude that $(\delta  d  + 2m)\tr  = 0$. 
This equation holds on all of $W$. Taking 
the $L^2$ dot product on $W$ of $\tr$
with this equation, we conclude that $\langle ( \delta  d  + 2m)\tr , \tr \rangle = 0.$ 
Integrating by parts gives
\begin{equation}\label{eqn:tr}
\langle  d \tr,  d \tr \rangle + 2m \langle \tr, \tr \rangle - 
\int_M~   \tr \wedge  {*}  d \tr ~=~ 0
\end{equation}
where the boundary is oriented using the outward normal.
Below we will show that the integral term of equation (\ref{eqn:tr}) 
equals $0$.  It will follow that $\tr = 0$
on all of $W$. 

Since ${\hat \omega}$ equals the 
model harmonic $\E$-valued $1$-form $\omega_0$ of Proposition \ref{pullback}  
in a neighborhood of the boundary, and since model forms all have trace zero, 
the trace 
of $\omega = {\hat \omega} + d_\E  s$ is just the trace of 
$d_\E  s$ in a neighborhood of the boundary.  
In turn, this is the
divergence of the vector field $u$ which is the translational 
part of $s$;
this equals $-\delta \tau$, where 
$\tau$ is the $1$-form dual to $u$.

Since $\tr = -\delta \tau$ in a neighborhood of the boundary, we can
use the boundary conditions on $\tau$
when computing the boundary integral. Also, since ${\hat \omega}$ is harmonic in
a neighborhood of the boundary, $\delta_\E  {\hat \omega} = 0$ in a neighborhood
of the boundary which in turn implies that $\delta_\E  d_\E  s$  is zero near the boundary.
This implies that
$(\lap + 2m) \tau = 0$ near the boundary.

The term $ {*}  d \tr$ in the integrand becomes $- {*} d  \delta \tau$.
Because 
$\tau$ satisfies $( d  \delta +  \delta  d + 2m) \tau = 0,$
this equals $ {*}( \delta  d + 2m) \tau$.  Since the integral is over the
boundary, only the tangential part of the integrand appears.  The tangential
part of ${*} \delta  d \tau =  d {*} d \tau$ equals
$d_M$ of the tangential part of ${*} d \tau$, 
where $d_M$ is the exterior derivative on the boundary.  But, the tangential part 
of $ {*}  d \tau$ equals ${*} (dr \wedge \widetilde \sigma)$, where the normal part of $ d \tau$ 
equals $dr \wedge \widetilde \sigma$.  However, boundary condition (\ref{formbdcond3}) implies
$\widetilde \sigma = 0$ on $M$.
Similarly, the tangential part of $ {*} \tau = {*} (h dr + \sigma)$ equals $h$ times  
the volume form of $M$. 
We know $h=0$ on $M$ 
by (\ref{formbdcond1}).  Thus, the boundary integral is zero and
we conclude that the infinitesimal deformation $\omega$ has trace zero, as desired. 

Finally, we need to show that, if $\omega$ is harmonic, has trace zero, and comes from 
a canonical lift, then $T^{*} \omega = 0$ (hence $D^{*} \omega = 0$ since $\delta_\E  \omega = 0$).
This is really an algebraic fact arising from the basic identities coming from being harmonic
and a canonical lift.

Because $\omega$ is a canonical lift, its translational part is symmetric 
when viewed as a section of $\Hom(TW,TW)$.  But 
$T^{*}(\sum_{i,j} a_{ij} E_i \otimes \omega_j) = \sum_{i,j} (a_{ji} -a_{ij}) R_{ij}$
so it is trivial on symmetric elements of $\Hom(TW,TW)$.
Next we need to compute $T^{*}$ on the rotational part of $\omega$.  Since $d_\E  \omega = 0$
it is locally the image under $d_\E$ of a section of $\E$.  Expressing this section as a pair 
$(u,\widetilde u)$ consisting of its translational  and rotational parts, we see that the rotational 
part of $\omega$ equals $Tu + D\widetilde u$.  Thus, we want to compute 
$T^{*}T u + T^{*}D \widetilde u$.

One computes that $T^{*}T(E_i) = \sum_k [E_k,[E_k,E_i]] = m E_i$, so that 
$T^{*}T(u) = m u$.  To compute $T^{*} D \widetilde u$, recall that being a canonical
lift means that the skew symmetric transformation $\widetilde u$ corresponds to the $2$-form
$\frac {1}{2} d \tau$, where $\tau$ is the $1$-form dual to $u$. 
Below we will see that the vector field obtained by applying
$T^{*}D$ to a section of $\so(TW)$ is dual to the $1$-form obtained by applying $\delta$
to the corresponding $2$-form.  Assuming this, we have that $T^{*} D \widetilde u$ is dual
to $\frac {1}{2} \delta d \tau$.  We conclude that $T^{*}$ of the rotational
part of $\omega$ is dual to $\frac {1}{2} (\delta d + 2 m) \tau$.  Since
$\omega$ is harmonic this equals $- \frac {1}{2} d \delta \tau$.  But 
$-\delta \tau $ equals the divergence of $u$ which, by definition, equals the 
trace of $\omega.$  Therefore $-\delta \tau = 0$ on 
all of $W$ so that $- \frac {1}{2} d \delta \tau = 0$, as desired.

To see the relation between $\delta$ and $T^{*} D$, we recall 
that, under the identification between sections of $\so(TW)$ and $2$-forms, the element $R_{ij}$
corresponds to $- \omega_i \wedge \omega_j$.  
We first compute $T^{*}D(a R_{ij})$ where $a$ is a function.  
From (\ref{DTeqns}) we conclude that 
$D (a R_{ij}) = \sum_k \del_{e_k}(a R_{ij}) \otimes \omega_k.$
and that $T^{*}(\del_{e_k}(a R_{ij}) \otimes \omega_k) = [E_k, \del_{e_k}(a R_{ij})].$ 
Thus we obtain 
\[T^{*} D (a R_{ij}) = \sum_k [E_k, \del_{e_k}(a R_{ij})] = 
(\del_{e_i} a) E_j - (\del_{e_j} a) E_i + a \sum_k [E_k, \del_{e_k}(R_{ij})].\]
Similarly, using a standard expression for $\delta$ on forms, we have
\[ - \delta (a \omega_i \wedge \omega_j) = 
\sum_k i(e_k) \del_{e_k}(a \omega_i \wedge \omega_j) =  
(\del_{e_i} a) \omega_j - (\del_{e_j} a) \omega_i 
+ a \sum_k i(e_k) (\del_{e_k}(\omega_i \wedge \omega_j)).\]
The first two terms of these two expressions are clearly dual.  At any point we can choose
a frame and coframe using geodesic coordinates and see that all the terms $\del_{e_k}(R_{ij})$ 
and $\del_{e_k}(\omega_i \wedge \omega_j)$ are zero.  Since $\delta$ and $T^*D$ are independent 
of this choice, we see that they are always dual.  

\section{Computing the \W boundary term} \label{sec:bdryterm}

We now assume that we
have a closed and co-closed form $\omega$ satisfying
$T^{*} \omega = 0$.  
The 
boundary term (\ref{Weitzbdry}) in the \W formula 
is then 
the integral of 
$- {*}T\omega \wedge \omega$ over
the boundary, oriented by the outward normal.  If this term is
non-negative, $\omega$ must be zero.  We have found it easier to keep
track of signs by writing 
the boundary integrand as $(-1)^m \omega \wedge {*}T\omega$.
We will give a general formula for this boundary integrand. 
This requires some preliminary notation.  
We will see that when $s$ is a solution to (\ref{correctionprob}) satisfying the boundary 
conditions of the previous section, and we let $\omega = \hat \omega + d_\E s$, 
the boundary integral (\ref{Weitzbdry}) will equal $0$.

We first describe the matrix corresponding to
$\omega \in \Omega^1(W,\E)$
in terms of $8$ sub-matrices
or blocks that are determined by dividing the columns into $2$
groups, the rows into $4$ groups, and taking all 
possible combinations.  
The columns are indexed by the $1$-forms $dr, \omega_1,\omega_2, \ldots, \omega_m$ corresponding to our coframe.  They are grouped according
to normal and tangential directions.  (So the first ``subset" 
consists of the single element $dr$.). 
The first $m+1$ rows correspond to the infinitesimal translations $N,E_1,E_2, \ldots, E_m$.  Again, they are grouped according to normal (single element)
and tangential ($m$ elements) directions.  The next subset of rows 
correspond to the infinitesimal rotations $R_{ab}$ that do not
involve the normal direction.  There are $k$ such elements, 
where $k$ is the dimension of $\so(m)$.  These can be put in any order; they 
will not play much of a role. 
The final group of $m$ rows 
correspond to the infinitesimal rotations $R_{ni}$ that do involve
the normal direction.

The resulting $8$ blocks are labeled by letters $\bA - \bH$ in the following
way, where the labeling has been chosen so that the most important terms
come first.  (The case when $m=3$ is pictured in Table \ref{1forms}.)
$\bA$ is the $m \times m$ block of elements of the form
$E_i \otimes \omega_j$ and $\bB$ is $m\times m$ block whose
entries are the coefficients of $R_{ni} \otimes \omega_j$.  They have already made
an appearance and are our main players.
Then there is an $m \times 1$ block with coordinates
$R_{ni} \otimes dr$ and a $k \times 1$ block with
coordinates $R_{ab} \otimes dr$.  These are labeled $\bC$
and $\bD$ respectively.  These are followed by a block whose 
entries are the coefficients of $E_i \otimes dr$ ($m \times 1$, labeled $\bE$) and 
the coefficient of $N \otimes dr$ ($1 \times 1$ block, labeled $\bF$).  The remaining
two blocks with tangential forms have entries that are the coefficients of $N \otimes \omega_j$
($1 \times m$) and of $R_{ab} \otimes \omega_j$ ($k \times m$), respectively
and are labeled $\bG$ and $\bH$.

\begin{table}
\[\kbordermatrix{		& dr	    &	    	& \omega_1	& \omega_2  & \omega_3	\\
                N	& \bF	    &\vrule	&		& \bG	    &		\\ \cline{2-6}
                E_1	&	    &\vrule	&		&	    &		\\
                E_2	& \bE	    &	\vrule	&		& \bA	    &		\\
		E_3	&	    &	\vrule	&		&	    &		\\ \cline{2-6}
		R_{12}	&	    &	\vrule	&		&	    &		\\
		R_{13}	& \bD	    &	\vrule	&		& \bH	    &		\\
		R_{23}	&	    &	\vrule	&		&	    &		\\ \cline{2-6}
		R_{n1}	&	    &	\vrule	&		&	    &		\\
		R_{n2}	& \bC	    &	\vrule	&		& \bB	    &		\\
		R_{n3}	&	    &	\vrule	&		&	    &		
}\]
\caption{The blocks of $\omega$}\label{1forms}
\end{table}

\bigskip

We want to compute $\omega \wedge {*}T \omega$.  Since we will be integrating
this form over the boundary, we are only interested in the tangential
terms of the forms. So the formula will only involve the tangential 
part $\omega_M$ of $\omega$ and the normal part of $T \omega$.  
(The ${*}$ operator interchanges the forms with a $dr$ term and those without.) 
The tangential 
$1$-forms have a basis given by $\{\omega_i\}$, and the normal
$2$-forms have a natural basis given by $\{dr \wedge \omega_j\}$.  
The decomposition of $\E$ into $4$ groups decomposes the matrix of normal $2$-forms into $4$ blocks.  These blocks have the same shapes as $\bG$, $\bA$, $\bH$ and $\bB$, respectively. 

Note that 
$\omega_i \wedge {*}(dr \wedge \omega_j)$ equals $0$ when 
$i \neq j$ and is a positive orientation of $M$ when $i = j$.
It follows that the integral of $\omega \wedge {*}T \omega$ over $M$ 
equals the $L^2$ dot product on $M$ of the $1$-forms $\omega_M$ and 
$i(n) T\omega$, where $i(n)$ denotes the interior product with the outward
normal.  The form $i(n) T\omega$ is simply the normal part of $T \omega$ without the $dr$ term, so we can continue to speak of the $4$ blocks of $i(n) T \omega$. 

The $L^2$ product of $\omega_M$ and 
$i(n) T\omega$ on $M$ equals the integral of the function 
obtained by taking
the dot products of each of $\bG$, $\bA$, $\bH$, and $\bB$ with its
corresponding block in $i(n) T \omega$ and then adding them together.

This leaves the task of computing the blocks of $i(n) T\omega$.  
For $m=3$ the answer is shown in Table \ref{2forms}.  The next few paragraphs 
explain this computation. 
There will be terms of 
two different types: those coming from the tangential part $\omega_M$ of $\omega$ 
and those coming from the normal part.  Looking at formula 
(\ref{DTeqns}) for $T$,  
we see that this can be expressed as
\begin{equation}\label{eqn:inTw}
i(n) T\omega = [N, \omega_M] - T_M (i(n) \omega).
\end{equation}
Here $T_M$ is the operator $T$ restricted to $M$; the sum in formula 
(\ref{DTeqns}) is taken only over the vectors $\{ e_1 , \ldots, e_m \}$ tangent to $M$.

\begin{table}
\[
\kbordermatrix{		& dr & & \omega_1 & \omega_2 & \omega_3\\
                N	& 0 &\vrule & & \bC^T & \\ \cline{2-6}
                E_1	&  &\vrule & & & \\
                E_2	& 0 &\vrule & & \bB+ \text{skew}(\bD)	& \\
		E_3	&  &\vrule & & & \\ \cline{2-6}
		R_{12}	&  &\vrule & & & \\
		R_{13}	& 0 &\vrule & & \text{rot}(\bE) &\\
		R_{23}	&  &\vrule & & & \\ \cline{2-6}
		R_{n1}	&  &\vrule & & & \\
		R_{n2}	& 0 &\vrule & & \bA + \bF (\text{Id}) & \\
		R_{n3}	&  &\vrule & & & 
}\]
\caption{The $1$-form $i(n) T \omega$}\label{2forms}
\end{table}

Taking the bracket of $N$ with $\omega_M$ 
simply interchanges the translational Lie algebra elements $\{E_i\}$
with the normal rotational ones $\{R_{ni}\}$ in the obvious way and annihilates the
others.  This has the effect of interchanging the blocks labelled $\bA$ and $\bB$.  In
other words, the first summand of (\ref{eqn:inTw}) consists of a copy of
$\bB$ in the second block, a copy of $\bA$ in the fourth block, and zeros elsewhere.  
This is visible in Table \ref{2forms}. 

To understand the contribution from the 
second summand of (\ref{eqn:inTw}), we simply drop
the normal $1$-form $dr$ from the first column $c$ of $\omega$ and think of $c$ as a section of $\E$.  We then compute $- T_M c$.
We can describe the contribution to the first block ($1 \times m$) as $\bC^T$, 
to the second block ($m \times m$) as
$\skewm(\bD)$, the third block ($k \times m$) as
$\rotm(\bE)$, and the fourth block ($m \times m$) as
$\bF(Id)$. 
In these descriptions $\bC^T$ is just the transpose of the
block $\bC$ in $\omega$, and $\bF (Id)$ is a scalar matrix. 
The block $\bD$ belongs to $\so(m) \otimes dr$ so that, dropping the $dr$ term, 
it represents an element of $\so(m).$  Then $\skewm(\bD)$ denotes this skew symmetric 
transformation, where we are identifying $TM$-valued $1$-forms with sections of
$\Hom(TM,TM)$ in the usual way.  
Dropping the $dr$ part of the block $\bE$, it becomes
simply a  vector in $TM$.  Its image under $T$ is a $1$-form with values in the 
bundle of tangential infinitesimal rotations, 
defined by sending $E_i$ to $\sum_j R_{ij} \otimes \omega_j$
and extending linearly.  We define $\rotm(\bE)$ to be the image of the vector $\bE$ under
this map.  In low dimensions this operator is related to the cross-product 
\cite{HK3}, which motivates our notation.  However, we don't have a 
good geometric description in general and use it simply as a definition here. 

Adding these two summands of (\ref{eqn:inTw}) we obtain the result shown in Table \ref{2forms} (for $m=3$). 
With this computation and notation we can now write down the integrand
$\omega \wedge {*}T \omega$ on $M$.
We record the answer in the following 
proposition.

\begin{prop} \label{integrand}
Let $\omega$ be an $\E$-valued $1$-form on an oriented hyperbolic manifold $W$ 
of dimension $n = m+1$ with totally geodesic boundary $M$.  Then, 
with respect to the notation above, the real-valued $1$-form
$\omega \wedge {*}T\omega$ restricted to the boundary equals the volume
form on $M$ times the function
\begin{equation} \label{bdterm}
\bG \cdot \bC^T + \bA \cdot \bB + \bA \cdot \skewm(\bD) + \bH \cdot \rotm(\bE) +\bB \cdot \bA + \bB \cdot \bF(Id)
\end{equation}
\end{prop}

In the situation of interest to us $\omega = \omega_0 + d_\E s$ where $\omega_0$
is a model harmonic form near 
the boundary and 
$d_\E s$ is a correction term.  We will use the letters $\hat \bA - \hat \bH$ to denote 
the blocks of $\omega_0$ and $\bA - \bH$ to denote those of 
$d_\E s$.
The form $\omega_0$ is particularly simple with all the blocks
equal to $0$ except for $\hat \bB$ which is symmetric and traceless.
In the correction term, $s$ is a global
section of $\E$. In particular, it is a global
section on $M$. Furthermore, the section $s$
is 
a canonical lift so that $\bA$ will be
the symmetric
part of $D v$ where $v$ is a global vector field on $M$.

Below we will prove the following lemma. 

\begin{lem} \label{0blocks}
Let $d_\E  s$ be an $\E$-valued $1$-form, where $s$ is a global section
of $\E$ which is a canonical lift satisfying 
boundary conditions (\ref{vfbdcond1}) and (\ref{vfbdcond2}).  Denote the blocks of $d_\E$ by the letters
$\bA - \bH$ as in Table \ref{1forms} above.  Then $\bB$, $\bD$, $\bE$, and $\bG$ are all zero. 
\end{lem}

Assuming Lemma \ref{0blocks}, we can show that, for a harmonic form $\omega$ obtained 
from our boundary value problem, the boundary term in the \W formula (\ref{Weitzformula})
is always trivial.  This implies infinitesimal, hence local, rigidity. 

\begin{thm} \label{boundaryzero}
Let $W$ be a compact hyperbolic $n$-manifold with totally geodesic boundary $M$, where
$n \geq 4$.  Let $\omega \in H^1(W;\E)$ be a harmonic representative constructed using
the boundary value problem in Section \ref{sec:bdryvalue}.  Then the boundary term 
\begin{equation*}
- \int_M \left( {*}T\omega \wedge \omega + T^{*} \omega \wedge {*} \omega \right) .
\end{equation*}
in the \W formula (\ref{Weitzformula}) equals $0.$   Hence $\omega = 0$.
\end{thm}

\begin{proof}
Since $T^{*} \omega = 0$, it suffices to show that the integral of 
$\omega \wedge {*}T \omega$ over $M$ is trivial.  We use Proposition \ref{integrand}
and the notation above to compute the integrand of this integral.

Using the decomposition $\omega = \omega_0 + d_\E s$, we
expand to obtain
\[\omega \wedge {*}T \omega = \omega_0 \wedge {*}T \omega_0 + d_\E s \wedge {*}T d_\E s + \omega_0 \wedge {*}T d_\E s + d_\E s \wedge {*}T \omega_0 \]
Each of the four terms is a sum of dot products 
with the same terms as those in equation (\ref{bdterm}) except that 
some of the letters have hats. 
In the first term all letters have hats, in the second none do, in the third the first letter in each dot product
has a hat, 
and in the fourth the second letter in each dot product
has a hat. 

From Lemma \ref{0blocks} we know that blocks $\bB$, $\bD$, $\bE$ and $\bG$ are all zero.  In $\omega_0$ every block except $\hat \bB$ is zero. 
One quickly sees that the first two terms are 
trivial because at least one matrix in each dot product equals $0$.  The third term
equals $\hat \bB \cdot (\bA + \bF (Id)) = \hat \bB \cdot \bA + \bF (\tr \hat \bB) = \hat \bB \cdot \bA$ since $\hat \bB$ is traceless.  The
final term equals $\bA \cdot \hat \bB$.  Thus the integrand
equals $ 2 \bA \cdot \hat \bB$.  Integrating this over $M$
is just $2$ times the $L^2$ dot product $\langle \bA, \hat \bB \rangle$ 
of the forms
$\bA$ and $\hat \bB$ (where $\hat \bB$ is interpreted as a $TM$-valued
$1$-form on $M$). 
Since $\hat \bB$ is symmetric, its dot product with $Dv$, for $v$ any 
global vector field on $M$, equals its dot product with the 
symmetric part of $Dv$ (this is true pointwise, just a fact
about dot products of matrices).  Since the section $s$ is a canonical lift, 
$\bA$ is the symmetric part of $Dv$ for a global vector field $v$ on $M$. 
So, taking the $L^2$ dot product,
$\langle\cdot,\cdot\rangle_M$, on $M$, we obtain
\[\langle \bA, \hat \bB \rangle_M = \langle Dv, \hat \bB \rangle_M = \langle v, D^{*} \hat \bB \rangle_M. \]
By Proposition \ref{pullback}, $0 = D^{*} \omega_0 = D^{*} \hat \bB$. 
We conclude that integrating $\bA \cdot \hat \bB$ over $M$ produces $0$.
\end{proof}

We conclude this section with the proof of Lemma \ref{0blocks}.

\begin{proof}


Recall that $\del_n e_i = 0 = \del_n n$ in a neighborhood of $M.$ Since 
$M$ is totally geodesic, $\del_{e_i} n = 0$ and $\del_{e_i} e_j$ has no normal component. The tangential components of
$\del_{e_i} e_j$ depend on the choice of moving frame, but are not relevant to our computations.

Near the boundary the bundle $\so(TW)$ decomposes into
infinitesimal rotations $\{R_{ab}\}_{n \notin \{a,b\}}$ fixing $n$, called tangential rotations, and those of the form $R_{ni}$, which we will call normal rotations.  The above discussion 
about covariant derivatives implies 
that, on 
$M$, $\del_{e_i} R_{ab}$ involves only tangential rotations and $\del_{e_i} R_{nj}$ involves
only normal rotations.  The normal derivatives of all rotations are zero.

These observations show that the exterior differentiation operator
$D: \Omega^0(M;\E) \to \Omega^1(M;\E)$ preserves a decomposition of the
domain and range into blocks in the following way.  Near the boundary, 
an element 
of the bundle $\E$ is decomposed into $4$ blocks 
corresponding to the normal translational part, tangential translational part,
tangential rotational part, and normal rotational 
part, respectively.  Similarly,
a matrix representing an $\E$-valued $1$-form is decomposed into the $8$ blocks 
of Table \ref{1forms}. 
$D$ preserves these decompositions in the sense that the only contribution
to a normal or tangential block in $\Omega^1(M;\E)$ comes from normal or tangential
derivatives respectively of the corresponding block in $\E$.  For example,
the contribution from $Ds$ to the block $ \bB$ whose coordinates are of the
form $R_{ni} \otimes \omega_j$ come from tangential derivatives of the normal rotational
block of $s$. 

With these preliminary observations, it is easy 
to see the 
effects of our 
boundary conditions on 
the blocks of $Ds$.  
As before, let $s = (u, \widetilde u)$ for $u$ a vector field on $W$ and $\widetilde u$ a section of $\so (TW)$. 
Applying boundary condition (\ref{vfbdcond1}) we see that, on $M$, 
$(\del_{e_i} u)\cdot n = \del_{e_i} (u \cdot n) = \del_{e_i} h = 0$ because $h \equiv 0$
and where we have used that $\del_{e_i} n = 0$ on $M$.  This is the contribution from $Ds$ to
block $ \bG$, whose coordinates are $N \otimes \omega_i$.
Similarly, because the $e_i$ are parallel in the normal direction, the boundary
condition (\ref{vfbdcond2}) implies that
$(\del_n u)\cdot e_i = \del_n (u \cdot e_i) = 0$. This is the $ \bE$ block of $Ds$.


Since $s = (u,\widetilde u)$ is a canonical lift, its rotational part, $\widetilde u$, is 
the skew symmetric part of $Du$ and corresponds to the $2$-form $ \frac {1}{2}  d \tau$,
where $\tau$ is dual to $u$.  The normal rotational part of $\widetilde u$, 
which is a linear combination of the $R_{ni}$, 
is just the part of the $2$-form $ \frac {1}{2}  d \tau$ 
involving $dr$.  It equals
$\frac {1}{2} \sum_i ((\del_n u)\cdot e_i - (\del_{e_i} u)\cdot n)\, dr \wedge \omega_i.$
We have just shown that
this is zero, so we conclude that the normal rotational part of
$s$ is trivial on $M$.  Since this is true on all of $M$, its tangential derivative is
zero.  We saw above that this is precisely the image in $ \bB$ of $Ds$.

Finally, the image in $ \bD$ of $Ds$ comes from the normal derivative
of the tangential rotational part of $s$.  The tangential rotational part of $s$
is described by the coefficients of $R_{ab}$.  Since 
$s$ is a canonical lift
the coefficient of $R_{ab}$ equals $- \frac{1}{2} ((\del_{e_a} u)\cdot e_b - (\del_{e_b} u)\cdot e_a).$
We don't know anything about these values except that, since $\del_{e_a} n = 0 = 
\del_{e_b} n$, they equal
$- \frac{1}{2} ((\del_{e_a} v)\cdot e_b - (\del_{e_b} v)\cdot e_a),$ where $v$ is the tangential
part of the vector field $u$.  However, since it is totally geodesic, $\del_{e_i}$
and $\del_n$ commute on $M$ and the normal derivative of these tangential derivatives
of $v$ can be expressed as tangential derivatives of the normal derivative of $v$.
But, by boundary condition (\ref{vfbdcond2}), these are all zero.
Thus, we have shown that, under our boundary conditions, $Ds$ has trivial image in the blocks
$ \bB$, $ \bD$, $ \bE$, and $ \bG$.  

Since $d_\E  = D + T$ we also need to compute
the value of $T(s)$ in these blocks.  We can quickly compute that 
$T(N) = - \sum_i R_{ni} \otimes \omega_i$, so its image is in
$ \bB$, and 
$T(E_i) = R_{ni} \otimes dr + \sum_j R_{ji} \otimes \omega_j$ 
with image in $ \bH$ and $ \bC$.  Similarly, 
$T(R_{ab}) = E_b \otimes \omega_a - E_a \otimes \omega_b$ is contained
in $ \bA$ and $T(R_{ni}) = E_i \otimes dr - N \otimes \omega_i$
contributes to $ \bE$ and $ \bG$.  However, for the section $s$ the 
coefficient of $N$ and all the coefficients of the $R_{ni}$ are zero 
(by boundary 
conditions (\ref{formbdcond1}) and (\ref{formbdcond3}), 
respectively) so
the value of $T s$ in each of $ \bB$, $ \bE$, and $ \bG$ is trivial. 
The block $ \bD$ is not in the image of $T$.
Thus, we have shown that, under our boundary conditions, 
the $4$ blocks $ \bB$, $ \bD$, $ \bE$, and $ \bG$ of $ds$ are all trivial.  
\end{proof}

This completes the proof of Theorem \ref{thm:main}.

\section{Extensions and Conjectures}\label{sec:conjectures}

In this section we will discuss some fairly immediate extensions of the results in this paper and  
some conjectures suggested by our methods.    

A first obvious question to ask is whether these results generalize to the case of manifolds 
with parabolic elements.  
The situation turns out to be more subtle than one might expect at first glance.  As we will
discuss below, the answer depends on whether or not the cusps are on the boundary.
Furthermore, we conjecture that the answer is different in dimension $4$ than 
in higher dimensions.

Let $W$ be a finite volume hyperbolic $n$-manifold with totally geodesic boundary.  In Theorem \ref{thm:intromain} $W$ is assumed to be compact, but in general it might have cusps and hence be noncompact.  
If the boundary of $W$ is nonetheless compact, 
then the methods used in this paper 
extend easily to prove local rigidity.  The main thing to check is that, when the dimension 
$n$ of $W$ is at least $4$, then the cusped ends remain complete under any small deformation.  
This follows from the fact that all of the cusps will have rank $n-1$ and a simple analysis of 
the algebraic deformations of rank $n-1$ parabolic subgroups 
in $\h^n.$   Once it is established that the cusps remain complete, standard $L^2$ techniques 
allow one to deal with the \W boundary terms of a harmonic representative in 
$H^1(X;\E)$ on these ends.  The rest of our analysis goes through without change.  
This proves the following theorem. 
 
\begin{thm}\label{thm:rank3}
Let $W$ be a finite volume hyperbolic $n$-manifold with compact, totally geodesic boundary. Assume $n>3$. 
Then the holonomy representation of $W$ is infinitesimally rigid.
\end{thm}

On the other hand, if the boundary of $W$ is allowed to have cusps, 
hence to be noncompact, then in dimension $4$ it is no longer true that $W$ is always locally rigid.  
The first example  of a such a flexible, finite volume hyperbolic $4$-manifold with 
non-compact geodesic boundary is constructed and studied in detail in \cite{KS}.  (Actually, the 
example is an orbifold, but it has manifold covers with similar properties.)

It is worth looking at how the arguments of the current paper break down in that case, as it provides
some insight into the situation in other dimensions.  If we denote by $W$ the finite volume hyperbolic manifold and let $M$ denote its geodesic boundary, then we consider the restriction map $H^1(W;\E) \to H^1(M;\E).$
The image of this map corresponds to those infinitesimal deformations in $\h^n$ of the $n-1$ dimensional 
hyperbolic manifold $M$ that extend over $W$.  We have seen that
there is a decomposition 
\begin{equation} \label{cohdecomp} 
H^1(M;\E) \cong H^1(M;\mathcal{H}) \oplus H^1(M;\mathcal{S}).
\end{equation}
When $M$ is compact and has dimension at least $3$, the first factor is trivial.
This was crucial to our argument.  When $M$ is complete and finite volume but noncompact, 
the factor $H^1(M;\mathcal{H})$ is \emph{nontrivial} when $M$ has dimension $3.$
This corresponds to the fact that such an $M$ has local deformations where the new
structure is no longer complete; in particular, at least one cupsed end will no longer
be complete.  This phenomenon is the basis for hyperbolic Dehn surgery which is an
important tool in the study of hyperbolic $3$-manifolds.

The analysis in this paper strongly suggests that, for any nontrivial element of $H^1(W;\E)$, 
the image in $H^1(M;\mathcal{H})$ under the restriction map must be nontrivial.  
Indeed, this is the case for the example in \cite{KS}.  In that example some of the cusps 
on the (orbifold) boundary of the original structure do not remain complete and 
undergo an 
orbifold version of hyperbolic Dehn surgery.  Furthermore, nearby representations
of the boundary groups no longer preserve a $3$-dimensional totally geodesic
hyperplane; in other words the boundary of convex hull is no longer totally geodesic.
This means that the image in the other factor, $H^1(M;\mathcal{S})$ is also 
nonzero.  
We conjecture
that this must always be true, that a nontrivial deformation of $W$ must always have
nontrivial image in \emph {both} factors in (\ref{cohdecomp}) under the restriction map.

Similar reasoning leads us to suspect that this failure of local rigidity
only occurs in dimension $4$.  For, when the dimension $n$ of $W$ is at least $5$,
the dimension of the boundary $M$ is at least $4$.  In these dimensions, 
the results of Garland-Raghunathan \cite{GR} imply that $H^1(M;\mathcal{H}) = 0$
whenever $M$ is complete and finite volume, even in the noncompact case.

\begin{conj}\label{conj:dimn5}
Let $W$ be a finite volume hyperbolic $n$-manifold with totally geodesic boundary.
If $n>4$ then $W$ is infinitesimally rigid.
\end{conj}

We have necessarily restricted ourselves to hyperbolic manifolds of dimension at least
$4$ in proving local rigidity results.  Our results are simply false in dimension $3$,
where infinite volume convex cocompact manifolds have large deformation spaces
corresponding to the Teichm{\"u}ller spaces of their conformal boundaries.
However, our analysis does provide some information about infinitesimal 
deformations even in dimension $3$.

Suppose $W$ is $3$-dimensional with totally geodesic boundary $M$ equal to 
a finite collection of
closed surfaces of genus $g\geq 2.$  In this dimension both factors of the decomposition 
 (\ref{cohdecomp}) of $H^1(M;\E)$ are nontrivial.  In fact they are actually isomorphic to
each other and the harmonic representatives can be identified with holomorphic
quadratic differentials on $M$.  The boundary value problem discussed in this paper is solvable
in the same way as before.  Analysis of the \W boundary term shows that any nontrivial
deformation must have nontrivial image under the restriction map in both of the factors 
of  $H^1(M;\E) \cong H^1(M;\mathcal{H}) \oplus H^1(M;\mathcal{S}).$  This means that, 
not only can't the boundary remain totally geodesic (which is clear 
by doubling and applying local rigidity in the closed case), but its hyperbolic metric 
must also change infinitesimally.  Viewed as quadratic differentials the two images of 
the restriction map must have a nontrivial $L^2$ pairing on $M.$  This leads to the 
following theorem:

\begin{thm}\label{thm:3dim}
Let $W$ be a compact hyperbolic $3$-manifold with totally geodesic boundary $M$ equal to a
collection of hyperbolic surfaces of genus $g \geq2.$  If $\omega \in H^1(W;\E)$ is a nontrivial 
infinitesimal deformation of $W$ then its image under the restriction map in 
$H^1(M;\E) \cong H^1(M;\mathcal{H}) \oplus H^1(M;\mathcal{S})$ is nontrivial in \emph{both} factors.  
Furthermore, when the two factors are suitably identified with holomorphic quadratic 
differentials, their $L^2$ inner product on $M$ is positive. 
\end{thm}

With this collection of theorems and conjectures we have begun to develop a picture in
all dimensions of the deformation theory of an $n$-dimensional hyperbolic manifold with totally 
geodesic boundary.  The increasing rigidity of the boundary as the dimension increases 
is reflected in the rigidity of the manifold itself.  Before this, the authors were unable to 
discern any significant structure in the deformation theory of high dimensional 
hyperbolic manifolds, and it seemed possible that the rigidity or flexibility in the infinite volume case was just a chaotic phenomenon.  At one extreme there are flexible (convex cocompact) 
examples with free fundamental group; they are always very flexible due to the lack of relations.  At the other extreme lie the closed manifolds; they are always rigid both locally and globally by Mostow rigidity \cite{Mos}.  However, little was known about cases in between.  At least now we have a new class of 
examples that are consistently locally rigid.

However, it remains unclear what the implications of these results are for the deformation theory of general higher dimensional convex cocompact groups.  If one considers 
an infinite volume complete manifold $X$ with compact convex core, the condition 
that the convex core have totally geodesic boundary implies certain topological properties
of $X.$
For example, it implies that the fundamental 
group of the boundary of the convex core injects into the fundamental group of $X$ and that $X$ has the homotopy type of an $(n-1)$-complex.
This is in contrast to the case when the fundamental group is free, where the 
holonomy representation is very flexible.  In this case $X$ is homotopy equivalent to a $1$-dimensional complex.  
Whether or not these topological properties are central remains to be seen.  

In dimensions $n \geq 4$ it is very difficult to construct examples from which one can formulate 
a conjectural picture.  In particular, if we consider the smooth manifold underlying 
an $n$-dimensional hyperbolic manifold $X$ with Fuchsian ends, we can ask whether it has other 
locally rigid hyperbolic structures or whether there are deformable ones. We are unable to answer this question because we do not know
if there even are any other hyperbolic structures.  It is quite possible that the hyperbolic manifold $X$ satifies an infinite volume version of Mostow rigidity.  While the results in this paper suggest that this is plausible,
other methods than those used here would be needed to make progress on this issue.
We conclude by raising it as a question.

\begin{quest}\label{quest:Mostow}
Let $f: X_1 \to X_2$ be a homotopy equivalence between complete hyperbolic $n$-manifolds without boundary, where $n>3$.  Assume that $X_1$ has Fuchsian ends, is convex cocompact, and is not Fuchsian. Is $f$ necessarily homotopic to an isometry? 
\end{quest}

\bibliography{bibliography}
\end{document}